\documentclass[12pt]{amsart}
\usepackage{a4wide,enumerate,graphicx}
\allowdisplaybreaks

\let\pa\partial  
  
\let\eps\varepsilon  
\newcommand{\N}{{\mathbb N}}  
\newcommand{\R}{{\mathbb R}}

\newtheorem{theorem}{Theorem}

\begin{document}  

\title[A kinetic equation modeling irrationality and herding]{A kinetic equation 
for economic value estimation with irrationality and herding}

\author{Bertram D\"uring}
\address{Department of Mathematics, University of Sussex, Pevensey II, 
Brighton BN1 9QH, United Kingdom}
\email{b.during@sussex.ac.uk} 
\author{Ansgar J\"ungel}
\address{Institute for Analysis and Scientific Computing, Vienna University of  
	Technology, Wiedner Hauptstra\ss e 8--10, 1040 Wien, Austria}
\email{juengel@tuwien.ac.at} 
\author{Lara Trussardi}
\address{Institute for Analysis and Scientific Computing, Vienna University of  
	Technology, Wiedner Hauptstra\ss e 8--10, 1040 Wien, Austria}
\email{lara.trussardi@tuwien.ac.at}

\date{\today}

\thanks{The first author is supported by the Leverhulme Trust research project 
grant ``Novel discretisations for higher-order nonlinear PDE'' (RPG-2015-69).
The second and third authors acknowledge partial support from   
the European Union in the FP7-PEOPLE-2012-ITN Program under 
Grant Agreement Number 304617 and from
the Austrian Science Fund (FWF), grants P22108, P24304, and W1245} 

\begin{abstract}
A kinetic inhomogeneous Boltzmann-type equation is proposed to model the
dynamics of the number of agents in a large market depending on the estimated
value of an asset and the rationality of the agents. The interaction rules
take into account the interplay of the agents with sources of public information, 
herding phenomena, and irrationality of the individuals.
In the formal grazing collision limit, a nonlinear nonlocal Fokker-Planck
equation with anisotropic (or incomplete) diffusion is derived. The existence
of global-in-time weak solutions to the Fokker-Planck initial-boundary-value
problem is proved. Numerical experiments for the Boltzmann equation
highlight the importance of the reliability of public information in the
formation of bubbles and crashes. The use of Bollinger bands in the simulations
shows how herding may lead to strong trends with low volatility of the asset prices,
but eventually also to abrupt corrections.
\end{abstract}

%\paragraph{Keywords:}  
 
%\paragraph{AMS classification:}  

\dedicatory{Dedicated to Peter A. Markowich on the occasion
of his $60$th birthday}

\maketitle

%%%%%%%%%%%%%%%%%%%%%%%%%%%%%%%%%%%%%%%%%%%%%%%%%%%%%%%%%%%%%%%%%%%%%%%%%%%%%%%

\section{Introduction}\label{sec.int}

\noindent Herding behavior and the formation of speculative bubbles (and subsequent crashes)
are observed in many financial and commodity markets. There are many historical
examples, from the so-called Dutch tulip bulb mania in 1637 to the recent
credit crunch in the US housing market in 2007. Despite the obvious importance of 
these phenomena, herding behavior and bubble formation were investigated
in the scientific literature only in the last two decades. 
The aim of this paper is to propose and investigate 
a kinetic model describing irrationality and herding of market agents,
motivated by the works of Toscani \cite{Tos06} and Delitala and Lorenzi \cite{DeLo14}.

Herding in economic markets is characterized by a homogenization of the actions
of the market participants, which behave at a certain time in the same way.
Herding may lead to strong trends with low volatility of asset prices, but
eventually also to abrupt corrections, so it promotes the occurence of
bubbles and crashes. Numerous socio-economic papers 
\cite{Ban92,Bru01,DJH09,RCF09,Roo06} and research in biological sciences
\cite{ARN05,Ham71} show that herding interactions play a crucial role in social
scenarios. Herding behavior is driven by emotions and usually occurs because 
of the social pressure of conformity. Another cause is the appeal to belief
that it is unlikey that a large number of people could be wrong. A market
participant might follow the herd in spite of another opinion. This phenomenon
is known as an information cascade \cite{BHW92}. 

While most approaches to herding in the literature are based on agent models,
our approach uses techniques from kinetic theory, similar to
opinion-formation models \cite{BS09,DMPW09,Tos06}. These methods employ ideas
from statistical mechanics to describe the behavior of a large number of
individuals in a society \cite{PaTo14}. Binary collisions between gas molecules are
replaced by interactions of market individuals, and the phase-space variables
are interpreted as socio-economic variables, in our case: the rationality $x\in\R$ 
and the estimated asset value $w\in\R^+:=[0,\infty)$, assigned to the asset
by an individual. When $x>0$, we say that the agent behaves
rational, otherwise irrational. We refer to the review \cite{DeWe96}
for a discussion of rational herding models.

Denoting by $f(x,w,t)$ the distribution of the agents at time $t\ge 0$,
its time evolution is given by the inhomogenous Boltzmann-type equation
\begin{equation}\label{1.be}
  \pa_t f + (\Phi(x,w)f)_x = \widetilde Q_I(f) + \widetilde Q_H(f,f),
  \quad (x,w)\in\R\times\R^+,\ t>0,
\end{equation}
with the boundary condition $f=0$ at $w=0$ and initial condition $f=f_0$ at $t=0$.
The first term on the right-hand side describes an interaction 
that is soley based on economic fundamentals. After the interaction, 
the individuals change their estimated asset value influenced by sources of public
information such as financial reports, balance sheet numbers, etc. 
The second term describes binary interactions of the agents modeling the 
exchange of information and possibly leading to herding and imitation phenomena. 

When the asset value lies within a certain range around the ``fair'' prize, 
determined by fundamentals, 
the agents may suffer from psychological biases like overconfidence and 
limited attention \cite{Hir01}, and we assume that they behave more irrational.
This means that the drift field $\Phi(x,w)$ is negative in that range.
When the asset value becomes too low or too large compared to
the ``fair'' prize the asset values are believed
to be driven by speculation. We assume that the market agents recognize
this fact at a certain point and are becoming more rational. 
In this case, the drift field $\Phi(x,w)$ is positive.
We expect that the estimated asset value will in average be not too far from the
``fair'' price, and we confirm this expectation by computing the moment
of $f(x,w,t)$ with respect to $w$ in Section \ref{sec.be}.
For details on the modeling, we refer to Section \ref{sec.mod}. 

Our setting is influenced by the models investigated by Toscani \cite{Tos06} and
Delitala and Lorenzi \cite{DeLo14}. Toscani \cite{Tos06} described the
interaction of individuals in the context of opinion formation.
Our modeling of public information and herding
is similar to \cite{Tos06}. The idea to include public information and herding
is due to \cite{DeLo14}. In contrast to \cite{DeLo14}, we allow
for the drift field $\Phi(x,w)$, leading to the inhomogeneous Boltzmann-type equation
\eqref{1.be}. Such equations were also studied in \cite{DuWo15} 
but using a different drift field.
The relationship of rational herd behavior and asset values was investigated
in \cite{AvZe98} but no dynamics were analyzed.
The novelty of the present work is the combination of dynamics, transport, 
public information, and herding.

Our main results are as follows.
We derive formally in the grazing collision limit (as in \cite{Tos06})
the nonlinear Fokker-Planck equation
\begin{align}\label{1.eq}
  & \pa_t g + (\Phi(x,w)g)_x = (K[g]g)_w + (H(w)g)_w + (D(w)g)_{ww}, \\
  & g(x,0,t) = 0, \quad g(x,w,0) = g_0(x,w), 
	\quad (x,w)\in \R\times\R^+,\ t>0.
  \label{1.bic}
\end{align}
Here, $K[g]$ is a nonlocal operator
related to the attitude of the agents to change their mind
because of herding mechanisms, $H(w)$ is an average of the compromise propensity,
and $D(w)$ models diffusion, which can be interpreted as a self-thinking process,
and satisfies $D(0)=0$. Again, we refer to Section \ref{sec.mod} for details.
A different herding diffusion model in the context of crowd motion was derived and
analyzed in \cite{BMP11}. Other kinetic and macroscopic crowd models were
considered in \cite{DAMT13}.

Equation \eqref{1.eq} is nonlinear, nonlocal, degenerate in $w$, 
and anisotropic in $x$ (incomplete diffusion). It is well known that
partial diffusion may lead to singularity formation \cite{HPW13},
and often the existence of solutions can be shown only in the class of
very weak or entropy solutions \cite{AmNa03,EVZ94}. 
Our situation is better than in \cite{AmNa03,EVZ94}, since the transport in $x$
is linear. Exploiting the linear structure, we prove the existence of global
{\em weak} solutions to \eqref{1.eq}-\eqref{1.bic}. However, we need the
assumption that $D(w)$ is strictly positive to get rid of the degeneracy in $w$.
Unfortunately, our estimates depend on $\inf D(w)$ and become useless when
$D(0)=0$. 
%The existence analysis is our second main result; see Section \ref{sec.anal}.

Finally, we present some numerical experiments for the inhomogeneous Boltzmann-type
equation \eqref{1.be} using a splitting scheme. 
The collisional part (i.e.\ \eqref{1.be} with $\Phi=0$)
is approximated using the interaction rules and a modified Bird scheme. The transport
part (i.e.\ \eqref{1.be} with $\widetilde Q_I=\widetilde Q_H=0$) is discretized
using a combination of an upwind and Lax-Wendroff scheme.
The numerical experiments 
highlight the importance of the reliability of public information in the
formation of bubbles and crashes. The use of Bollinger bands in the simulations
shows how herding may lead to strong trends with low volatility of the asset prices,
but eventually also to abrupt corrections.

The paper is organized as follows. In Section \ref{sec.mod}, the kinetic
model is detailed and the grazing collision limit is performed. The resulting
Fokker-Planck model \eqref{1.eq}-\eqref{1.bic} is analyzed in Section \ref{sec.anal}.
Furthermore, we discuss the time evolution of the moments of $g(x,w,t)$ 
in some specific examples.
The numerical results are presented in Section \ref{sec.num}.

%%%%%%%%%%%%%%%%%%%%%%%%%%%%%%%%%%%%%%%%%%%%%%%%%%%%%%%%%%%%%%%%%%%%%%%%%%%%%%%

\section{Modeling}\label{sec.mod}

\noindent The aim of this section is to model the evolution of the distribution of
the number of agents in a large market using a kinetic approach.

\subsection{Public information and herding}\label{sec.pih}

We describe the behavior of the market agents by means of microscopic interactions
among the agents. The state of the market is assumed to be characterized by
two continuous variables: the estimated asset value $w\in\R^+:=[0,\infty)$ and
the rationality $x\in\R$. We say that the agent has a rational behavior
if $x>0$ and an irrational behavior if $x<0$. The changes in asset valuation
are based on binary interactions. We take into account two different types:
the interaction with public sources, which characterizes a rational agent,
and the effect of herding, characterizing an irrational agent.
In the following, we define the corresponding interaction rules.

Let $w$ be the estimated asset value of an arbitrary agent before the interaction
and $w^*$ the asset value after exchanging information with the public source.
Given the background $W=W(t)$, which may be interpreted as a ``fair'' value,
the interaction is given, similarly as in \cite{CDT09}, by
\begin{equation}\label{2.inter1}
   w^* = w - \alpha P(|w-W|)(w-W)+\eta d(w).
\end{equation}
The function $P$ measures the compromise propensity and takes values in $[0,1]$,
and the parameter $\alpha>0$ is a measure of the strength of this effect. 
Furthermore, the function $d$ with values in $[0,1]$
describes the modification of the asset value due to diffusion, and
$\eta$ is a random variable with distribution $\mu$ with
variance $\sigma_I^2$ and zero mean taking values on $\R$, i.e.\
$\langle w\rangle=\int_\R wd\mu(w)=0$ and $\langle w^2\rangle 
= \int_\R w^2 d\mu(w)=\sigma_I^2$.
An example for $P$ is \cite{Tos06}
$$
  P(|w-W|) = 1_{\{|w-W|<r\}}, 
$$
where $r>0$ and $1_A$
denotes the characteristic function on the set $A$. Thus, if
the estimated asset value is too far from the value available from public
sources (the ``fair'' value), 
the effect of public information will be discarded (selective perception). 
The idea behind \eqref{2.inter1} is that if a market agent trusts an information
source, she will update her estimated asset value to make it closer to the one
suggested by the public information. We expect that a rational investor
follows such a strategy.

The interaction rule \eqref{2.inter1} has to ensure that the post-interaction 
value $w^*$ remains in the interval $\R^+$.
We have to require that diffusion vanishes at the border $w=0$,
i.e.\ $d(0)=0$. In the absence of diffusion, it follows that
$w^*=w-\alpha P(|w-W|)(w-W)\ge w-\alpha(w-W) = (1-\alpha)w + \alpha W\ge 0$ if $w>W$
and $w^*=w+\alpha P(|w-W|)(W-w)\ge w\ge 0$ if $w\le W$. Therefore,
the post-interaction value $w^*$ stays in the domain $\R^+$. 

The second interaction rule aims to model the effect of herding, i.e., we take
into account the interaction between a market agent and other investors.
We suggest the interaction rule, similarly as in \cite{Tos06},
\begin{equation}\label{2.inter2}
\begin{aligned}
  w^* &= w - \beta\gamma(v,w)(w-v) + \eta_1 d(w), \\
	v^* &= v - \beta\gamma(v,w)(v-w) + \eta_2 d(v).
\end{aligned}
\end{equation}
The pairs $(w,v)$ and $(w^*,v^*)$ denote the asset values of two arbitrary 
agents before and after the interaction, respectively. In \eqref{2.inter2}, 
$\beta\in(0,1/2]$ is a constant which measures the attitude of the market 
participants to change their mind because of herding mechanisms. Furthermore,
$\eta_1$, $\eta_2$ are random variables, modeling diffusion effects, 
with the same distribution with variance $\sigma_H^2$ and zero mean,
and, to simplify, the function $d$ is the same as in \eqref{2.inter1}.
The function $\gamma$ with values in $[0,1]$ describes a socio-economic 
scenario where individuals are highly confident in the asset.
An example, taken from \cite{DeLo14}, reads as
\begin{equation}\label{2.gamma}
  \gamma(v,w) = 1_{\{w<v\}}v f(w), 
\end{equation}
where $f$ is nonincreasing, $f(0)=1$, and $\lim_{w\to\infty}f(w)=0$.
If an agent has an asset value $w$ smaller than $v$, the function $\gamma$ 
will push this
agent to assume a higher value $w^*$ than that one before the interaction.
This means that the agent trusts other agents that assign a higher value.
If $w$ is larger than $v$, the agent hesitates to lower his asset value
and nothing changes. Agents that assign a small value $w$ tend to herd with
a higher rate, i.e.\ $f$ is nonincreasing. Another choice is given by
$\gamma(v,w)=1_{\{|w-v|<r_H\}}$ \cite{DeLo14}. In this case, the
interaction occurs only when the two interacting agents have asset values
which are not too different from each other.

The interaction does not take place if $w^*$, $v^*$ are negative. 
In the absence of diffusion, adding both equations in \eqref{2.inter2} gives
$w^*+v^*=v+w$ which means that the total momentum is conserved. Subtracting
both equations in \eqref{2.inter2} yields $w^*-v^*=(1-2\beta\gamma(v,w))(w-v)$.
Since $1-2\beta\gamma(v,w)\in[0,1)$ (observe that $0<\beta\le 1/2$),
the post-interaction difference $w^*-v^*$ in the asset values is smaller than the
pre-interaction difference $w-v$. We infer that $w^*$, $v^*$ remain nonnegative.

When diffusion is taken into account, we need to specify the range of values
the random variables $\eta_1$, $\eta_2$ in \eqref{2.inter2} can assume. This clearly
depends on the choice of $d(w)$, and we refer to \cite[page 3691]{DMPW09}
for a more detailed discussion.

%%%%%%%%%%%%%%%

\subsection{The kinetic equation}\label{sec.be}

Instead of calculating the value $x$ and $w$ for each market agent, we prefer
to investigate the evolution of the distribution $f(x,w,t)$ of the estimated value
and the rationality of the market participants.
The integral $\int_B f(x,w,t)dz$ with $z=(x,w)$ represents the number of agents
with asset value and rationality in $B\subset\R\times\R^+$ at time $t\ge 0$.
In analogy with classical kinetic theory of rarefied gases, we may identify
the position variable with the rationality and the velocity with the asset value.
Using standard methods of kinetic theory, $f(x,w,t)$ evolves according to the
inhomogeneous Boltzmann equation
\begin{equation}\label{2.be}
  \pa_t f + (\Phi(x,w)f)_x = \frac{1}{\tau_I}Q_I(f) + \frac{1}{\tau_H}Q_H(f,f), 
	\quad (x,w)\in\R\times\R^+,\ t>0.
\end{equation} 
Here, $\Phi(x,w)$ is the drift term, $Q_I$ and $Q_H$ are interaction integrals
modeling the public information and herding, respectively, and 
$1/\tau_I>0$, $1/\tau_H>0$ describe the interaction
frequencies. This equation is supplemented by the boundary condition
$f(x,0,t)=0$ (nobody believes that the asset has value zero) and the initial
condition $f(x,w,0)=f_0(x,w)$ for $(x,w)\in\R\times\R^+$.

A simple model for $\Phi$ can be introduced as follows. If an agent gives an asset
value that is much larger than the ``fair'' value $W$,
she will recognize that the value is overestimated and it is believed that she
will become more rational. The same holds true when the estimated value is too
low compared to $W$. In this regime, the drift function $\Phi(x,w)$ should be
positive since agents drift towards higher rationality $x>0$.
When the estimated value is not too far from the value $W$, agents may
behave more irrational and drift towards the region $x<0$, so the drift 
function is negative. An example for such a function is 
\begin{equation}\label{2.Phi}
  \Phi(x,w) = \left\{\begin{array}{ll}
	-\delta\kappa &\quad\mbox{for }|w-W|<R, \\
	\kappa &\quad\mbox{for }|w-W|\ge R,
	\end{array}\right.
\end{equation}
where $\delta$, $\kappa$, $R>0$. The constant $R$ fixes the range $|w-W|<R$ 
in which bubbles and crashes do not occur. More realistic models are
obtained when $R$ depends on time, and we consider such a case in Section 
\ref{sec.num}). An alternative is to employ the mean asset value 
$\int_\R\int_{\R^+}fwdwdx$ instead of $w$ in $|w-W|<R$ to distinguish the ranges.

Next, we detail the choice of the interaction integrals. 
As pointed out in \cite{CDT09}, the existence of a pre-interaction pair which returns 
the post-interaction pair $(w^*,v^*)$ through an interaction of the type 
\eqref{2.inter1} is not guaranteed, because of the boundary constraint.
Therefore, we will give the interaction rule in the weak form. 
Let $\phi(w):=\phi(x,w)$ be a regular test function 
and set $\Omega=\R\times\R^+$, $z=(x,w)$. 
The weak form reads as
\begin{equation}\label{2.QI}
  \int_\Omega Q_I(f)\phi(w)dz
	= \left\langle\int_{\R^+}\int_\Omega\big(\phi(w^*)-\phi(w)\big)M(W)f(x,w,t)
	dzdW\right\rangle,
\end{equation}
where $\langle\cdot\rangle$ is the expectation value with respect to the
random variable $\eta$ in \eqref{2.inter1} and $M(W)\ge 0$ represents the
fixed background satisfying $\int_{\R^+}M(W)dW=1$. 
The Boltzmann equation for this operator, $\pa_t f=Q_I(f)/\tau_I$, becomes
in the weak form
$$
  \pa_t\int_\Omega f(x,w,t)\phi(w)dz 
	= \frac{1}{\tau_I}
	\bigg\langle\int_{\R^+}\int_\Omega\big(\phi(w^*)-\phi(w)\big)M(W)f(x,w,t)
	dzdW\bigg\rangle.
$$
Choosing $\phi(w)=1$, the right-hand side vanishes, which expresses
conservation of the number of agents:
$$
  \pa_t\int_\Omega f(x,w,t)dz = 0.
$$
The choice $\phi(w)=w$ gives the time evolution of the mean asset value
$m_w(f)=\int_\Omega fwdz$:
$$
  \pa_t m_w(f) = \frac{1}{\tau_I}\bigg\langle\int_{\R^+}\int_\Omega
	(w^*-w)M(W)f(x,w,t)dzdW\bigg\rangle 
	= -\alpha\int_\Omega H(w)f(x,w,t)dz,
$$
where
\begin{equation}\label{2.H}
  H(w) = \frac{1}{\tau_I}\int_{R^+}P(|w-W|)(w-W)M(W)dW.
\end{equation}
For instance, if $P=1$ and denoting by $\rho:=\int_\Omega fdz$ the (conserved)
number of agents, we obtain $H(w)=\tau_I^{-1}(w-\overline M)$, where
$\overline M:=\int_0^\infty WM(W)dW$, and 
$$
  \pa_t m_w(f) = -\frac{\alpha}{\tau_I}
	\int_\Omega(w-\overline M)f(x,w,t)dz 
	= -\frac{\alpha}{\tau_I} m_w(f) + \frac{\alpha}{\tau_I}\rho\overline M.
$$
This shows that the mean asset value converges exponentially fast to the mean value
of the background as $t\to\infty$.

The operator $Q_H(f,f)$ models the binary interaction of the agents and,
similary as in \cite{Tos06}, we define
\begin{equation}\label{2.QH}
  \int_\Omega Q_H(f,f)\phi(w)dz
	= \bigg\langle\int_{\R^+}\int_\Omega\big(\phi(w^*)-\phi(w)\big)f(x,w,t)f(x,v,t)
	dzdv\bigg\rangle,
\end{equation}
where $(w,v)$ is the pre-interaction pair that generates via \eqref{2.inter2}
the post-interaction pair $(w^*,v^*)$. Choosing $\phi=1$ in the Boltzmann
equation $\pa_t f=Q_H(f,f)$, we see that this operator also conserves the number
of agents. Taking $\phi(w)=w$ and using a symmetry argument, 
the interaction rule \eqref{2.inter2},
and the fact that the random variables $\eta_1$ and $\eta_2$ have zero mean, 
we find that
\begin{align*}
  \pa_t m_w(f) 
	&= \frac{1}{\tau_H}\bigg\langle\int_{\R^+}\int_\Omega(w^*-w)f(x,w,t)f(x,v,t)
	dzdv\bigg\rangle \\
	&= \frac{1}{2\tau_H}\bigg\langle\int_{\R^+}\int_\Omega(v^*+w^*-v-w)f(x,w,t)f(x,v,t)
	dzdv\bigg\rangle \\
	&=  \frac{1}{2\tau_H}\bigg\langle\int_{\R^+}\int_\Omega
	\big(\eta_1 d(w)+\eta_2 d(v)\big)f(x,w,t)f(x,v,t)dzdv\bigg\rangle = 0,
\end{align*}
We infer that herding conserves the mean asset value. This is reasonable as
the crowd may tend to any direction depending on the herding.

%%%%%%%%%%%%%%%%%%%%%%%%

\subsection{Grazing collision limit}\label{sec.lim}

The analysis of the Boltzmann equation \eqref{2.be} is rather involved, and
it is common in kinetic theory to investigate certain asymptotics leading
to simplified models of Fokker-Planck type. Our aim is to perform the formal
limit $(\alpha,\beta,\sigma_H^2,\sigma_I^2)\to 0$ 
(in a certain sense made precise below), where $\alpha$, $\beta$ appear in
the interaction rules \eqref{2.inter1} and \eqref{2.inter2} and 
$\sigma_H^2$, $\sigma_I^2$ are the variances of the random variables in these 
rules. The limit can be made rigorous using the techniques of \cite{CPT05,Tos06},
but we prefer to consider the formal limit only. In the following, we proceed along
the lines of \cite{CPT05,Tos06}.

Set $k=\beta/\alpha$, $t_s=\alpha t$, $x_s=\alpha x$, and introduce the
functions $g(x_s,w,t_s)=f(x,w,t)$ $\Phi_s(x_s,w)=\Phi(x,w)$. 
After the change of variables $(x,w)\mapsto (x_s,t_s)$ and setting $z_s=(x_s,w)$, 
the weak form of \eqref{2.be} reads as
\begin{align}
  \frac{\pa}{\pa t_s} & \int_\Omega g(x_s,w,t_s)\phi(w)dz_s
	+ \int_\Omega\frac{\pa}{\pa x_s}\big(\Phi_s(x_s,w)g(x_s,w,t_s)\big)\phi(w)dz_s 
	\nonumber \\
	&= \frac{1}{\alpha\tau_I}\int_\Omega Q_{I,s}(g)\phi(w)dz_s
	+ \frac{1}{\alpha\tau_H}\int_\Omega Q_{H,s}(g,g)\phi(w)dz_s, \label{2.aux}
\end{align}
where $Q_{I,s}(g)=Q_I(f)$, $Q_{H,s}(g,g)=Q_H(f,f)$ are defined in weak form in
\eqref{2.QI}, \eqref{2.QH}, respectively. In the following, we omit the
index $s$. 

Before performing the formal grazing collision limit,
we rewrite the first term on the right-hand side of \eqref{2.aux}.
By a Taylor expansion and the interaction rule \eqref{2.inter1}, we can write
\begin{align*}
  \phi(w^*)-\phi(w) &= \phi'(w)(w^*-w) + \frac12\phi''(\widetilde w)(w^*-w)^2 \\
	&= \phi'(w)\big(-\alpha P(|w-W|)(w-W) + \eta d(w)\big) \\
	&\phantom{xx}{}
	+ \frac{1}{2}\phi''(\widetilde w)\big(-\alpha P(|w-W|)(w-W) + \eta d(w)\big)^2,
\end{align*}
where $\widetilde w=\theta w^* + (1-\theta)w$ for some $\theta\in[0,1]$.
Inserting this expression into \eqref{2.QI}, observing that
$\langle\eta\rangle=0$, $\langle\eta^2\rangle=\sigma_I^2$, and taking into 
account definition \eqref{2.H} for $H(w)$, it follows that
\begin{align*}
  \frac{1}{\alpha\tau_I} & \int_\Omega Q_I(g)\phi(w)dz
	= -\frac{1}{\tau_I}\int_{\R^+}\int_\Omega\phi'(w)P(|w-W|)(w-W)M(W)g(x,w,t)
	dzdW \\
	&\phantom{xx}{}
	+ \frac{1}{2\tau_I}\int_{\R^+}\int_\Omega \phi''(\widetilde w)\bigg(\alpha P(|w-W|)^2
	(w-W)^2 + \frac{\sigma_I^2}{\alpha}d(w)^2\bigg)	M(W)g(x,w,t)dzdW \\
	&= -\frac{1}{\tau_I}\int_\Omega\phi'(w)H(w)g(x,w,t)dz + R(\alpha,\sigma_I) \\
	&\phantom{xx}{}
	+ \frac{1}{2\tau_I}\int_{\R^+}\int_\Omega \phi''(w)\bigg(\alpha P(|w-W|)^2
	(w-W)^2 + \frac{\sigma_I^2}{\alpha}d(w)^2\bigg)	M(W)g(x,w,t)dzdW,
\end{align*}
where
\begin{align*}
  R(\alpha,\sigma_I) 
	&= \frac{1}{2\tau_I}\int_{\R^+}\int_\Omega \big(\phi''(\widetilde w)-\phi''(w)\big)
	\bigg(\alpha P(|w-W|)^2(w-W)^2 
	+ \frac{\sigma_I^2}{\alpha}d(w)^2\bigg)	\\
	&\phantom{xx}{}\times M(W)g(x,w,t)dzdW.
\end{align*}
We wish to pass to the limit $\alpha\to 0$ and $\sigma_I\to 0$ such that
$\lambda_I:=\sigma_I^2/\alpha$ is fixed. It is proved in \cite[Section~4.1]{Tos06}
that $R(\alpha,\sigma_I)\to 0$ as $(\alpha,\sigma_I)\to 0$. Then
\begin{align*}
  \lim_{(\alpha,\sigma_I)\to 0} & \frac{1}{\alpha\tau_I} \int_\Omega Q_I(g)\phi(w)dz
	= \frac{1}{\tau_I}\int_\Omega\bigg(-\phi'(w)H(w)
	+ \frac{\lambda_I}{2}d(w)^2\phi''(w)\bigg)g(x,w,t)dz \\
	&= \int_\Omega\bigg((H(w)g)_w + \frac{\lambda_I}{2\tau_I}(d(w)^2g)_{ww}
	\bigg)\phi(w)dz,
\end{align*}
where in the last step we integrated by parts. The boundary integrals vanish
since $g=0$ at $w=0$ and $d(0)=0$ imply that 
$(d(w)^2g)_w|_{w=0} = d'(0)g|_{w=0}+d(0)g_w|_{w=0}=0$.

The limit $(\alpha,\sigma_H)\to 0$ in the last term of \eqref{2.aux} is performed
in a similar way. Using a Taylor expansion and \eqref{2.inter2}, we have
(with a different $\widetilde w$ as above)
\begin{align*}
  \phi(w^*)-\phi(w) &= \phi'(w)\big(-\alpha k\gamma(v,w)(v-w)+\eta_1 d(w)\big) \\
	&\phantom{xx}{}+ \frac12\phi''(\widetilde w)
	\big(-\alpha k\gamma(v,w)(v-w)+\eta_1 d(w)\big)^2.
\end{align*}
We insert this expansion into \eqref{2.QH} and find that
\begin{align*}
  \frac{1}{\alpha\tau_H} & \int_\Omega Q_H(g,g)dz
  = \frac{1}{\tau_H}\int_{\R^+}\int_\Omega\bigg(-k\gamma(v,w)(v-w)\phi'(w) \\
	&\phantom{xx}{}+ \frac12\bigg(\alpha k^2\gamma(v,w)^2(v-w)^2 
	+ \frac{\sigma_H^2}{\alpha}d(w)^2\bigg)\phi''(w)\bigg)g(x,v,t)g(x,w,t)dzdv \\
	&= \int_\Omega\bigg(-K[g](x,w,t)\phi'(w)
	+ \frac{\sigma_H^2}{2\alpha\tau_H}d(w)^2\rho\phi''(w)\bigg)g(x,w,t)dz \\
	&\phantom{xx}{}+ \frac{\alpha k^2}{2\tau_H}\int_{\R^+}\int_\Omega 
	\gamma(v,w)^2(v-w)^2 g(x,v,t)g(x,w,t)\phi''(w)dzdw,
\end{align*}
where we recall that $\rho=\int_\Omega fdz$ and we have set
$$
  K[g](x,w,t) = \frac{k}{\tau_H}\int_0^\infty \gamma(v,w)(v.w)g(x,v,t)dv.
$$
Keeping $\lambda_H=\sigma_H^2/\alpha$ fixed, the limit $(\alpha,\sigma_H)\to 0$ 
leads to
\begin{align*}
  \lim_{(\alpha,\sigma_H)\to 0} \frac{1}{\alpha\tau_H}\int_\Omega Q_H(g,g)dz
	&= \int_\Omega\bigg(-K[g](x,w,t)\phi'(w)
	+ \frac{\lambda_H}{2\tau_H} d(w)^2\rho\phi''(w)\bigg)g(x,w,t)dz \\
	&= \int_\Omega\big((K[g]g)_w 
	+ \frac{\lambda_H\rho}{2\tau_H} (d(w)^2g)_{ww}\big)\phi(w) dz.
\end{align*}

Therefore, the limit $(\alpha,\sigma_I,\sigma_H)\to 0$ in \eqref{2.aux} gives
\begin{align*}
  \int_\Omega\big(\pa_t g + (\Phi(x,w)g)_x\big)\phi dz
	&= \int_\Omega\bigg((H(w)g)_w + (K[g]g)_w \\
	&\phantom{xx}{}	+ \frac12\bigg(\frac{\lambda_I}{\tau_I}
	+ \frac{\lambda_H\rho}{\tau_H}\bigg)(d(w)^2g)_{ww}\bigg)\phi dz.
\end{align*}
Since $\phi$ is an arbitrary test function, this is the weak form of the
Fokker-Planck-type equation
\begin{equation}\label{2.fp}
  \pa_t g + (\Phi(x,w)g)_x = (K[g]g + H(w))_w + \frac12\bigg(\frac{\lambda_I}{\tau_I}
	+ \frac{\lambda_H\rho}{\tau_H}\bigg)(d(w)^2g)_{ww}
\end{equation}
for $(x,w)\in\R\times\R^+$, $t>0$. This equation is supplemented by the
boundary condition $g=0$ at $w=0$ and the initial condition $g(0)=g_0$ in
$\Omega$.

%%%%%%%%%%%%%%%%%%%%%%%%%%%%%%%%%%%%%%%%%%%%%%%%%%%%%%%%%%%%%%%%%%%%%%%%%%%%%%%

\section{Analysis}\label{sec.anal}

\noindent The aim of this section is to analyze the Fokker-Planck-type equation
derived in the previous section. To this end, we set
$$
  \Gamma(v,w) := \frac{k}{\tau_H}\gamma(v,w)(v-w), \quad
	D(w) := \frac12\bigg(\frac{\lambda_I}{\tau_I}	
	+ \frac{\lambda_H\rho}{\tau_H}\bigg)d(w)^2.
$$
Then \eqref{2.fp} simplifies to
\begin{equation}\label{2.eq}
  \pa_t g + (\Phi(x,w)g)_x = (K[g]g + H(w)g)_w + (D(w)g)_{ww}, \quad
	K[g] = \int_0^\infty\Gamma(v,w)g(v)dv.
\end{equation}

\subsection{Existence of weak solutions}\label{sec.ex}

We wish to show the existence of weak solutions to \eqref{1.bic}, \eqref{2.eq}
under the following hypotheses:

\begin{description}
\item[H1] $\Phi\in W^{2,\infty}(\Omega)$, $H\in W^{1,\infty}(\R^+)$,
$D\in W^{2,\infty}(\R^+)$, and there exists $\delta>0$ such that 
$D(w)\ge\delta>0$ for $w\in(0,\infty)$.
\item[H2] $\Gamma\in L^2((\R^+)^2)$, $\Gamma\ge 0$,  
%$\sup_{w>0}\int_0^\infty\Gamma(v,w)dv<\infty$, 
and $\Gamma_w(v,w)\le 0$ for all $v$, $w\ge 0$.
\item[H3] $g_0\in H^1(0,\infty)$ and $0\le g_0\le M_0$ for some $M_0>0$.
\end{description}

Then the main result reads as follows.

\begin{theorem}\label{thm.ex}
Let Hypotheses H1-H3 hold. Then there exists a weak solution $g$ to
\eqref{1.eq}-\eqref{1.bic} satisfying $0\le g(x,w,t)\le M_0e^{\lambda t}$
for $(x,w)\in\Omega$, $t>0$, where $\lambda>0$ depends on $\Phi$, $H$ and $D$,
and it holds $g\in L^2(0,T;H^1(\Omega))$, $\pa_t g\in L^2(0,T;H^1(\Omega)')$.
\end{theorem}

The idea of the proof is to regularize equation \eqref{1.eq} by adding 
a second-order derivative with respect to $x$, 
to truncate the nonlinearity, and to solve the equation 
in the finite interval $w\in(0,R)$. Then we pass to the deregularization limit.
Let $R>0$, $0<\eps<1$, $M>0$, set
$$
  K_M[g](x,w,t) = \int_0^R\Gamma(v,w)(g)_M(x,v,t)dv, \quad
	(g)_M = \max\{0,\min\{M,g\}\},
$$
where $g$ is an integrable function, and
introduce $\Omega_R=(-R,R)\times(0,R)$.
We split the boundary $\pa\Omega_R=\pa\Omega_D\cup\pa\Omega_N$ into two parts:
$\pa\Omega_D=\{(x,w):x\in[-R,R],\ w=0,R\}$ and 
$\pa\Omega_N=\{(x,w):x=\pm R,\ w\in(0,R)\}$. Finally, we set $g^+=\max\{0,g\}$.
Consider the approximated nonlinear problem
\begin{align}
  & \pa_t g + (\Phi(x,w)g^+)_x = \big((K_M[g]+H(w)+D'(w))g^+\big)_w + (D(w)g_w)_w 
	+ \eps g_{xx}, \label{ex.eq.eps} \\
%	& g = 0 \quad\mbox{on }\pa\Omega_R,\ t>0, \quad 
  & g=0\quad\mbox{on }\pa\Omega_D, \quad
	g_x = 0\quad\mbox{on }\pa\Omega_N, \quad
	g(x,w,0) = 0\quad\mbox{in }\Omega_R. \label{ex.bic.eps}
\end{align}
We introduce the space $H_D^1(\Omega_R)$ consisting of those functions 
$v\in H^1(\Omega_R)$ which satisfy $v=0$ on $\pa\Omega_D$, and we set
$H_D^{-1}(\Omega_R)=(H_D^1(\Omega_R))'$.

The weak formulation of \eqref{ex.eq.eps}-\eqref{ex.bic.eps} reads as:
For all $v\in L^2(0,T;H_D^1(\Omega_R))$,
\begin{align}
  \int_0^T\langle\pa_t g,v\rangle dt
	&= -\int_0^T\int_{\Omega_R}\Big((\Phi_x(x,w)g^++\Phi(x,w)g^+_x)v \label{ex.weak} \\
	&\phantom{xx}{}
	+ \big(K_M[g]+H(w)+D'(w)\big)g^+v_w + d(w)g_wv_w + \eps g_xv_x\Big)dzdt, \nonumber
\end{align}
where $\langle\cdot,\cdot\rangle$ is the dual product between 
$H_D^{-1}(\Omega_R)$ and $H_D^1(\Omega_R)$.

\begin{proof}[Proof of Theorem \ref{thm.ex}.]
We wish to apply the Leray-Schauder fixed-point theorem. For this, we split the
proof in several steps.

{\em Step 1: solution of the linearized problem.}
Let $T>0$, $\widetilde g\in L^2(0,T;L^2(\Omega))$, and $\eta\in[0,1]$.
We introduce the forms
\begin{align}
  a(g,v) &= \int_{\Omega_R}\big(\eta\Phi(x,w)g_xv
	+ D(w)g_wv_w + \eps g_xv_x\big)dz, \quad
	g,v\in H_D^1(\Omega_R), \label{ex.defa} \\
	F(v) &= -\eta\int_{\Omega_R}\big(\Phi_x(x,w)\widetilde g^+v
	+ (K_M[\widetilde g]+H(w)+D'(w))\widetilde g^+v_w\big)dz. \label{ex.deff}
\end{align}
Since $K_M[\widetilde g]$ is bounded, it is not difficult to see that
$a$ is bilinear and continuous on $H_D^1(\Omega_R)^2$ and $F$ is linear and
continuous on $H_D^1(\Omega_R)$. Furthermore, using Young's inequality
and $D(w)\ge\delta>0$, it follows that, for some $C_\eps>0$,
\begin{align*}
  a(g,g) &\ge \frac12\int_{\Omega_R}\big(\delta g_w^2 dz + \eps(g_x^2+g^2)\big)dz
	- (C_\eps+\eps)\int_{\Omega_R} g^2 dz \\ 
	&\ge \min\{\delta,\eps\}\|g\|_{H^1(\Omega_R)}^2 
	- (C_\eps+\eps)\|g\|_{L^2(\Omega_R)}^2.
\end{align*}
By Corollary 23.26 in \cite{Zei90}, there exists a unique solution 
$g\in L^2(0,T;H^1_D(\Omega_R))\cap H^1(0,T;$ $H^{-1}_D(\Omega_R))$ to 
\begin{equation}\label{ex.a}
  \langle \pa_t g,v\rangle + a(g,v) = F(v), \quad t>0, \quad g(0)=\eta g_0.
\end{equation}
This defines the fixed-point operator 
$S:L^2(0,T;L^2(\Omega_R))\times[0,1]\to L^2(0,T;L^2(\Omega_R))$,
$S(\widetilde g,\eta)=g$, where $g$ solves \eqref{ex.a}. 
This operator satisfies $S(\widetilde g,0)=0$. Standard arguments show that
$S$ is continuous (employing $H^1$ estimates depending on $\eps$). Since
$L^2(0,T;H^1_D(\Omega_R))\cap H^1(0,T;$ $H^{-1}_D(\Omega_R))$ is compactly
embedded in $L^2(0,T;L^2(\Omega_R))$, the operator is also compact. In order
to apply the fixed-point theorem of Leray-Schauder, we need to show uniform
estimates. 

{\em Step 2: $L^\infty$ bound.} Let $g$ be a fixed point of $S(\cdot,\eta)$,
i.e., $g$ solves \eqref{ex.a} with $\widetilde g=g$. We choose $v=g^-:=\min\{0,g\}\in 
L^2(0,T;H_D^1(\Omega_R))$ as a test function in \eqref{ex.a} and integrate
over $(0,t)$. Since $g^+g^-=0$ and $g^-(0)=g_0^-=0$, we have
$$
  a(g,g^-) = \int_{\Omega_R}\big(D(w)(g^-_w)^2 + \eps(g^-_x)^2 \big)dz
	\ge 0, \quad F(g^-)=0,
$$
which shows that
$$
  \frac12\int_{\Omega_R}g^-(t)^2 dz = \frac12\int_{\Omega_R}g^-(0)^2 dz
	- \int_0^t a(g,g^-)ds \le 0.
$$
This yields $g^-=0$ and $g\ge 0$ in $\Omega_R$, $t>0$. In particular, we may write
$g$ instead of $g^+$ in \eqref{ex.defa}-\eqref{ex.deff}.

For the upper bound, we choose the test function 
$v=(g-M)^+\in L^2(0,T;H_D^1(\Omega_R))$ in \eqref{ex.weak}, where $M=M_0e^{\lambda t}$
for some $\lambda>0$ which will be determined later. By Hypothesis H3, 
$(g-M)^+(0)=(g_0-M_0)^+=0$. Observing that
$\pa_t M=\lambda M$, $(g-M)(g-M)_w^+=\frac12[((g-M)^+)^2]_w$ and integrating
by parts in the integrals involving $K_M[g]+H(w)+D'(w)$, we find that
\begin{align*}
  \frac12 & \int_{\Omega_R}(g-M)^+(t)^2 dz 
	= -\lambda\int_0^t\int_{\Omega_R}M(g-M)^+ dzds \\
	&\phantom{xx}{}
	-\eta\int_0^t\int_{\Omega_R}\big(\Phi_x(x,w)((g-M)+M)+\Phi(x,w)(g-M)_x^+\big)
	(g-M)^+ dzds \\
  &\phantom{xx}{}
	- \eta\int_0^t\int_{\Omega_R}(K_M[g]+H(w)+D'(w))((g-M)+M)(g-M)_w^+ dzds \\
	&\phantom{xx}{}
	- \int_0^t\int_{\Omega_R}\big(D(w)((g-M)_w^+)^2 + \eps((g-M)_x^+)^2\big)dzds \\
	&= -\eta\int_0^t\int_{\Omega_R}(\Phi_x(x,w)+\lambda)M(g-M)^+ dzds
	- \eta\int_0^t\int_{\Omega_R}\Phi_x(x,w)((g-M)^+)^2 dzds \\
	&\phantom{xx}{}- \eta\int_0^t\int_{\Omega_R}\Phi(x,w)(g-M)_x^+(g-M)^+ dzds \\
	&\phantom{xx}{}
	+ \frac{\eta}{2}\int_0^t\int_{\Omega_R}\big(K_M[g]_w + H'(w) + D''(w)\big)
	((g-M)^+)^2 dzds \\
	&\phantom{xx}{}
	+ \eta\int_0^t\int_{\Omega_R}\big(K_M[g]_w + H'(w) + D''(w)\big)M(g-M)^+ dzds \\
	&\phantom{xx}{}
	- \int_0^t\int_{\Omega_R}\big(D(w)((g-M)_w^+)^2 + \eps((g-M)_x^+)^2\big)dzds.
\end{align*}
The third integral on the right-hand side can be estimated by Young's inequality,
\begin{align*}
  - \eta\int_0^t\int_{\Omega_R}\Phi(x,w)(g-M)_x^+(g-M)^+ dzds
	&\le \frac{\eta}{2\eps}\|\Phi\|^2_{L^\infty(\Omega)}
	\int_0^t\int_{\Omega_R}((g-M)^+)^2 dzds \\
	&\phantom{xx}{}+ \frac{\eps}{2}\int_0^t\int_{\Omega_R}((g-M)_x^+)^2 dzds.
\end{align*}
Then, collecting the integrals involving $M(g-M)^+$ and $((g-M)^+)^2$ and
observing that $\Gamma_w\le 0$ implies that $K_M[g]_w\le 0$, it follows that
\begin{align*}
  \frac12 & \int_{\Omega_R}(g-M)^+(t)^2 dz 
	\le \eta\int_0^t\int_{\Omega_R}\big(-\Phi_x(x,w)+H'(w)+D''(w)-\lambda\big)
	M(g-M)^+ dzds \\
	&\phantom{xx}{}
	+ \frac{\eta}{2}\int_0^t\int_{\Omega_R}\bigg(\frac{1}{\eps}
	\|\Phi\|^2_{L^\infty(\Omega)}-2\Phi_x(x,w)+H'(w)+D''(w)\bigg)((g-M)^+)^2 dzds \\
	&\phantom{xx}{}
	- \int_0^t\int_{\Omega_R}\bigg(D(w)((g-M)_w^+)^2 + \frac{\eps}{2}((g-M)_x^+)^2\bigg)
	dzds.
\end{align*}
Choosing $\lambda\ge \|\Phi_x\|_{L^\infty(\Omega)}+\|H'\|_{L^\infty(0,\infty)}
+\|D''\|_{L^\infty(0,\infty)}$, the first integral on the right-hand side
is nonpositive. The last integral is nonpositive too, and the second integral
can be estimated by some constant $C_\eps>0$. We conclude that
$$
  \int_{\Omega_R}(g-M)^+(t)^2 dz 
	\le C_\eps\int_0^t\int_{\Omega_R}((g-M)^+)^2 dzds.
$$
Then Gronwall's lemma implies that $(g-M)^+=0$ and $g\le M$ in $\Omega_R$, $t>0$.
In particular, we can write $K[g]$ instead of $K_M[g]$ in \eqref{ex.weak}.

The $L^\infty$ bound provides the desired bound for the fixed-point operator
in $L^2(0,T;$ $L^2(\Omega_R))$,
yielding the existence of a weak solution to \eqref{ex.weak}.

{\em Step 3: uniform $H^1$ bound.} We wish to derive $H^1$ bounds independent 
of $\eps$. To this end, we choose first the test function 
$v=g\in L^2(0,T;H_D^1(\Omega_R))$ in \eqref{ex.weak} (replacing $T$ by $t\in (0,T)$):
\begin{align*}
  \frac12\int_{\Omega_R}g(t)^2 dz
	&= -\int_0^t\int_{\Omega_R}\Phi_x(x,w)g^2 dzds
	- \int_0^t\int_{\Omega_R}\Phi(x,w)g_xg dzds \\
	&\phantom{xx}{}-\frac12\int_0^t\int_{\Omega_R}\big(K[g]+H(w)+D'(w)\big)
	(g^2)_w dzds \\
	&\phantom{xx}{}- \int_0^t\int_{\Omega_R}\big(D(w)g_w^2 + \eps g_x^2\big)dzds
	+ \frac12\int_{\Omega_R}g_0^2dz.
\end{align*}
Applying Young's inequality to the second integral on the right-hand side,
integrating by parts in the third integral, and observing
that $g=0$ at $w\in\{0,R\}$ yields, for some constant $C_1>0$ which depends on the
$L^\infty$ norms of $\Phi_x$, $H'$, and $D''$ (we use again that $K[g]_w\le 0$),
\begin{align}
  \frac12\int_{\Omega_R}g(t)^2 dzdt 
	&\le C_1\int_0^t\int_{\Omega_R} g^2 dzds
	+ C_1\int_0^T\int_{\Omega_R}g_x^2dzds \label{ex.ineq1} \\
	&\phantom{xx}{}- \int_0^T\int_{\Omega_R}\big(\delta g_w^2 + \eps g_x^2\big)dzds
	+ \frac12\int_{\Omega_R}g_0^2dz. \nonumber 
\end{align}
Since $C_1>\eps$ is possible, this does not give an estimate, and we need
a further argument. 

Next, we differentiate \eqref{ex.eq.eps} with respect to $x$ in the sense of
distributions and set $h:=g_x$:
\begin{align}\label{ex.h}
  \pa_t h + \big(\Phi_x(x,w)g + \Phi(x,w)h\big)_x 
	&= (K[h]g)_w + \big((K[g] + H(w) + D'(w))h\big)_w \\
	&\phantom{xx}{}+ (D(w)h_w)_w + \eps h_{xx} \quad\mbox{in }\Omega_R,\ t>0. \nonumber
\end{align}
We observe that the boundary condition
$g=0$ on $\pa\Omega_D$ implies that also
$g_x=0$ holds on $\pa\Omega_D$ and so, $g_x=0$ on $\pa\Omega_R$.
Hence, equation \eqref{ex.h} is complemented with homogeneous 
Dirichlet boundary conditions. Furthermore,
$h(x,w,0)=g_{0,x}(x,w)$. The weak formulation of \eqref{ex.h} reads as
\begin{align*}
  \int_0^T\langle \pa_t h,v\rangle dt
	&= -\int_0^T\int_{\Omega_R}\Big(\big(\Phi_{xx}(x,w)g + 2\Phi_x(x,w)h 
	+ \Phi(x,w)h_x\big)v \\
	&\phantom{xx}{}
	+ K[h]gv_w + \big(K[g]+H(w)+D'(w)\big)hv_w + D(w)h_wv_w + \eps h_xv_x\Big)dzdt
\end{align*}
for all $v\in L^2(0,T;H_0^1(\Omega_R))$. 
This is a linear nonlocal problem for $h$, with given $g$,
and we verify that there exists a solution
$h\in L^2(0,T;H^1_0(\Omega_R))\cap H^1(0,T;H^{-1}(\Omega_R))$, using
similar arguments as above. Therefore, 
we can choose $v=h$ as a test function in \eqref{ex.h}:
\begin{align*}
  \frac12 & \int_{\Omega_R}h(t)^2 dz
	= -\int_0^t\int_{\Omega_R}\bigg((\Phi_{xx}(x,w)gh + 2\Phi_x(x,w)h^2
	+ \frac12\Phi(x,w)(h^2)_x\bigg)dzds \\
	&\phantom{xx}{}- \int_0^t\int_{\Omega_R}\bigg(K[h]gh_w 
	+ \frac12\big(K[g]+H(w)+D'(w)\big)(h^2)_w + D(w)h_w^2 + \eps h_x^2\bigg)dzds \\
	&\phantom{xx}{}+ \frac12\int_{\Omega_R}g_x(0)^2dz.
\end{align*}
We integrate by parts and employ the inequalities $K[g]_w\le 0$, $D(w)\ge\delta$:
\begin{align}
  \frac12 & \int_{\Omega_R}h(t)^2 dzds
	\le -\int_0^t\int_{\Omega_R}\bigg(\Phi_{xx}(x,w)gh + \frac32 \Phi_{x}(x,w)h^2
	\bigg)dzds \nonumber \\
	&\phantom{xx}{}- \int_0^t\int_{\Omega_R}K[h]gh_w dzds
	+ \frac12\int_0^t\int_{\Omega_R}\big(H'(w)+D''(w)\big)h^2 dzds \label{ex.aux1} \\
	&\phantom{xx}{}- \int_0^t\int_{\Omega_R}\big(\delta h_w^2 + \eps h_x^2\big)dzds
	+ \frac12\int_{\Omega_R}g_x(0)^2dz. \nonumber
\end{align}
The first integral on the right-hand side is estimated by using Young's inequality:
\begin{align*}
  \int_0^t\int_{\Omega_R}\bigg(\Phi_{xx}(x,w)gh + \frac32 \Phi_{x}(x,w)h^2
	\bigg)dzds
	&\le \frac12\|\Phi_{xx}\|_{L^\infty(\Omega)}\int_0^t\int_{\Omega_R}(g^2+h^2)dzds \\
	&\phantom{xx}{}
	+ \frac32\|\Phi_x\|_{L^\infty(\Omega)}\int_0^t\int_{\Omega_R}h^2 dzds.
\end{align*}
For the second integral on the right-hand side of \eqref{ex.aux1}, 
we observe that $0\le g\le M$ and
$\|K[h]\|_{L^2(\Omega_R)}\le C_\Gamma\|h\|_{L^2(\Omega_R)}$, 
where $C_\Gamma^2=\int_0^\infty\int_0^\infty\Gamma(v,w)^2 dvdw$. Thus,
\begin{align*}
  - \int_0^t\int_{\Omega_R}K[h]gh_w dzds
	&\le M\int_0^t\|h\|_{L^2(\Omega_R)}\|h_w\|_{L^2(\Omega_R)}ds \\
	&\le \frac{\delta}{2}\int_0^t\int_{\Omega_R}h_w^2 dzds
	+ \frac{M}{2\delta}\int_0^t\int_{\Omega_R}h^2 dzds.
\end{align*}
This shows that, for some $C_2(\delta)>0$,
\begin{align}
  \frac12 \int_{\Omega_R}h(t)^2 dz
	&\le C_2(\delta)\int_0^t\int_{\Omega_R}(g^2+h^2)dzds
	- \frac{\delta}{2}\int_0^t\int_{\Omega_R}h_w^2 dzds \nonumber \\
	&\phantom{xx}{}- \eps\int_0^t\int_{\Omega_R}h_x^2 dzds
	+ \frac12\int_{\Omega_R}g_x(0)^2dz. \label{ex.ineq2}
\end{align}
We add \eqref{ex.ineq1} and \eqref{ex.ineq2} to find that, for some $C_3(\delta)>0$,
\begin{align*}
  \int_{\Omega_R}(g^2+h^2)(t)dz 
	+ \int_0^t\int_{\Omega_R}\big(\delta g_w^2 + \eps h^2 + \eps h_x^2\big)dzds
	&\le C_3(\delta)\int_0^t\int_{\Omega_R}(g^2 + h^2) dzds \\
	&\phantom{xx}{}+ \frac12\int_{\Omega_R}(g_0^2+g_{0,x}^2)dz.
\end{align*}
Gronwall's lemma provides uniform estimates for $g$ and $g_x=h$:
\begin{equation}\label{ex.unif1}
  \|g\|_{L^\infty(0,T;L^2(\Omega_R))}
	+ \|g_x\|_{L^\infty(0,T;L^2(\Omega_R))}
	+ \|g_w\|_{L^2(0,T;L^2(\Omega_R))} \le C,
\end{equation}
where $C>0$ depends on $\delta$, $M$, and the $L^{\infty}$ bounds for 
$\Phi$, $H$, $D'$ and their derivatives, but not on $R$ and $\eps$.

{\em Step 4: limit $\eps\to 0$.} We wish to pass to the limit $\eps\to 0$
in \eqref{ex.eq.eps}. Let $g_\eps:=g$ be a solution to 
\eqref{ex.eq.eps}-\eqref{ex.bic.eps} with $K[g]=K_M[g]$.
First, we estimate $\pa_t g_\eps$:
\begin{align}
  \|\pa_t g_\eps\|_{L^2(0,T;H_D^{-1}(\Omega_R))}
	&\le \|\Phi(x,w)g_\eps\|_{L^2(0,T;L^2(\Omega_R))} \nonumber \\
	&\phantom{xx}{}+ \|K[g_\eps]+H(w)+D'(w)\|_{L^\infty(0,T;L^\infty(\Omega_R))}
	\|g_\eps\|_{L^2(0,T;L^2(\Omega_R))} \label{ex.unif2} \\
	&\phantom{xx}{}+ \big(\|D\|_{L^\infty(0,T;L^\infty(\Omega_R))} + 1\big)
	\|g_\eps\|_{L^2(0,T;H^1(\Omega_R))} \le C, \nonumber
\end{align}
where $C>0$ does not depend on $\eps$ and $R$ (since $K[g_\eps]$ is uniformly
bounded). Estimates \eqref{ex.unif1} and
\eqref{ex.unif2} allow us to apply the Aubin-Lions lemma to conclude the
existence of a subsequence of $(g_\eps)$, which is not relabeled, such that
as $\eps\to 0$, 
\begin{align*}
  g_\eps\to g &\quad\mbox{strongly in }L^2(0,T;L^2(\Omega_R)), \\
	g_\eps\rightharpoonup g &\quad\mbox{weakly in }L^2(0,T;H^1(\Omega_R)), \\
	\pa_t g_\eps\rightharpoonup \pa_t g &\quad\mbox{weakly in }
	L^2(0,T;H_D^{-1}(\Omega_R)).
\end{align*}
By the Cauchy-Schwarz inequality, this shows that
\begin{align*}
  \| & K[g_\eps]-K[g]\|_{L^2(0,T;L^2(\Omega_R))} \\
	&\le \bigg(\int_0^R\int_0^\infty\Gamma(v,w)^2dvdw\bigg)
  \int_0^T\int_0^R\int_{-R}^R (g_\eps-g)^2(x,w,t)dxdwdt \\
	&\le C_\Gamma\|g_\eps-g\|_{L^2(0,T;L^2(\Omega_R))} \to 0 \quad\mbox{as }\eps\to 0.
\end{align*}
We infer that
$$
  K[g_\eps]g_\eps \to K[g]g \quad\mbox{strongly in }L^1(0,T;L^1(\Omega_R)).
$$
Since $(K[g_\eps]g_\eps)$ is bounded, this convergence holds in $L^p$ 
for any $p<\infty$.
Consequently, we may perform the limit $\eps\to 0$ in \eqref{ex.weak} 
(with $g^+=g$ and $K_M[g]=K[g]$) to obtain for all $v\in L^2(0,T;H_D^1(\Omega_R))$,
\begin{align}
  \int_0^T\langle\pa_t g,v\rangle 
	&= -\int_0^T\int_{\Omega_R}\Big((\Phi_x(x,w)g + \Phi(x,w)g_x)v \label{ex.weak2} \\
	&\phantom{xx}{}+ \big(K[g]+H(w)+D'(w)\big)gv_w + D(w)g_wv_w\Big)dzdt \nonumber
\end{align}

{\em Step 5: Limit $R\to\infty$.} The limit $R\to\infty$ is based on
Cantor's diagonal argument. We have shown that there exists a weak solution 
$g_n$ to \eqref{ex.weak2} with $g_n(0)=g_0$ in the sense of $H_D^{-1}(\Omega_n)$,
where $n\in\N$. In particular, $(g_n)$ is bounded in $L^2(0,T;H^1(\Omega_m))$ for all
$n\ge m$. We can extract a subsequence $(g_{n,m})$ of $(g_n)$ that converges weakly
in $L^2(0,T;H^1(\Omega_m))$ to some $g^{(m)}$ as $n\to\infty$. 
Observing that the estimates in Step 4 are independent of $R=n$, we obtain even
the strong convergence $g_{n,m}\to g^{(m)}$ in $L^2(0,T;L^2(\Omega_m))$ and a.e.\
in $\Omega_m\times(0,T)$. This yields the diagonal scheme
\begin{align*}
  g_{1,1},\ g_{2,1},\ g_{3,1},\ &\ldots\ \to g^{(1)} = u|_{\Omega_1\times(0,T)}, \\
	g_{2,2},\ g_{3,2},\ &\ldots\ \to g^{(2)} = u|_{\Omega_2\times(0,T)}, \\
	g_{3,3},\ &\ldots\ \to g^{(3)} = u|_{\Omega_3\times(0,T)}, \\
	&\ddots \phantom{\to xx} \vdots
\end{align*}
This means that there exists a subsequence $(g_{n,1})$ of $(g_n)$ that converges
strongly in $L^2(0,T;H^1(\Omega_1))$ to some $g^{(1)}$. From this subsequence,
we can select a subsequence $(g_{n,2})$ that converges strongly in 
$L^2(0,T;H^1(\Omega_1))$ to some $g^{(2)}$ such that $g^{(2)}|_{\Omega_1\times(0,T)}
= g^{(1)}$, etc. The diagonal sequence $(g_{n,n})$ converges to some $g\in
L^2(0,T;H^1(\Omega))$ which is a solution to \eqref{1.eq}-\eqref{1.bic}.
\end{proof}

%%%%%%%%%%%%%%%%%%%%%%%%%%%%%%%%%%%%%%%%%%%%%%%%%%%%%%%%%%%%%%%%%%%%%%%%%%%%%

\subsection{Asymptotic behavior of the moments}\label{sec.mom}

We analyze the time evolution of the macroscopic moments
$$
  m_w(g) = \int_\Omega g(x,w,t)wdz, \quad m_x(g) = \int_\Omega g(x,w,t)xdz,
$$
where $g$ is a (smooth) solution to \eqref{1.eq}-\eqref{1.bic}, in the special
situation that $P=1$ and $\Phi(x,w)$ is given by \eqref{2.Phi}. Observe
that $P=1$ implies that (recall \eqref{2.H})
$$
  H(w) = w-W, \quad\mbox{where }W=\int_0^\infty \omega M(\omega)d\omega.
$$
The parameter $W$ may be the same as in the definition of $\Phi(x,w)$ in
\eqref{2.Phi}.
First, we compute $\pa_t m_w(g)$. 
Using $g=0$ at $w=0$ and integrating by parts with respect to $w$, we obtain
$$  
  \pa_t m_w(g) = -\int_{\Omega}\big(K[g]g + H(w)g\big)dz
	\le -\int_\Omega(w-W)gdz = -m_w(g) + W\rho,
$$
where we have taken into account that $K[g]\ge 0$ and $\rho=\int_\Omega gdz$.
By Gronwall's lemma, $m_w(g(t))$ converges exponentially fast to the mean
value $\rho W$ as $t\to\infty$. 
This result is similar to the convergence of the mean
asset value for solutions to the Boltzmann equation, as shown in
Section \ref{sec.be}.

Next, we compute $\pa_t m_x(g)$. Then
$$
  \pa_t m_x(g) = \int_\R\Phi(x,w)g dw
	= -\delta\kappa\int_\R\int_{\{|w-W|<R}gdwdx + \kappa\int_\R\int_{\{|w-W|\ge R}gdwdx.
$$
This expression explains the role of the parameter $\delta$. Indeed, assume that
in some time interval, the number of agents with estimated asset value
around $W$ ($|w-W|<R$) is of the same order as those with asset value which differs
significantly from $W$ ($|w-W|\ge R$). Then, for $\delta\gg 1$, the
mean rationality is decreasing, and if $\delta\ll 1$, it is increasing. 
Thus, $\delta$ is a measure for the expected mean rationality.

The variance of $g$ with respect to $w$ is given by
$V_w(g)=\int_\Omega g(w-W)^2dz$. We compute
$$
  \pa_t V_w(g) = -2\int_\Omega(K[g]+H(w))(w-W)g dz + 2\int_\Omega D(w)g dz.
$$
At this point, we need some simplifying assumptions. Let $\Gamma(v,w)=\Gamma_0$ 
and $D(w)=w$. Then $K[g]=\Gamma_0\rho$ and
\begin{align*}
  \pa_t V_w(g) 
	&= -2\int_\Omega\big(\Gamma_0\rho(w-W)g + (w-W)^2g\big) dz
	+ 2\int_\Omega gwdz \\
	&= 2(1-\Gamma_0\rho)m_w(g)+ 2\Gamma_0\rho^2 W - 2V_w(g).
\end{align*}
We infer from $m_w(g(t))\to \rho W$ 
that the variance $V_w(g(t))$ converges to $2\rho W$ as $t\to\infty$.

%%%%%%%%%%%%%%%%%%%%%%%%%%%%%%%%%%%%%%%%%%%%%%%%%%%%%%%%%%%%%%%%%%%%%%%%%%%%%

\section{Numerical simulations}\label{sec.num}

\noindent  We illustrate the behavior of the solution to the kinetic model derived
in Section \ref{sec.be} by numerical simulations. 

\subsection{The numerical scheme}\label{sec.scheme}

The kinetic equation \eqref{2.be} is originally posed in the unbounded
spatial domain 
$(x,w)\in\R\times\R^+$. Numerically, we consider instead a bounded domain,
similarly as for the approximate equation \eqref{ex.eq.eps} in the existence
analysis. For this, let $(x,w)\in(-1,1)\times(0,1)$. This means that
agents with $x=-1$ are completely irrational and individuals with $x=1$ are
completely rational. The maximal possible asset value $w$ is normalized to one.
We choose uniform
subdivisions $(x_0,\ldots,x_N)$ for the variable $x$ and $(w_0,\ldots,w_M)$
for the variable $w$. We take $N=M=70$ in the simulations. 
The function $f(x,w,t_k)$ is approximated by $f_{ij}^k$, where $x\in(x_i,x_{i+1})$,
$w\in(w_j,w_{j+1})$, and $t_k=k\triangle t$, 
where $\triangle t$ is the time step size (we choose $\triangle t=10^{-5}$).

For the numerical approximation, we make an operator splitting ansatz,
i.e., we split the Boltzmann equation \eqref{2.be} into a collisional part
and a drift part. The collisional part 
$$
  \pa_t f = Q_I(f) \quad\mbox{or}\quad \pa_t f = Q_H(f,f)
$$
is numerically solved by using the interaction rules
\eqref{2.inter1} or \eqref{2.inter2}, respectively, 
and a slightly modified Bird scheme \cite{Bir94}.
First, we describe the choice of the interaction rule. 
The stochastic process $\eta$ is a point process with $\eta=\pm 0.06$
with probability 0.5. 
The total number of agents is normalized to one. 
We introduce the number of irrational agents
$I_{\rm irr}(w,t)=\int_{-1}^0 f(x,w,t)dx$ and the number of rational agents
$I_{\rm rat}(w,t)=\int_0^1 f(x,w,t)dx$. If for fixed $(w,t)$, the majority 
of the agents is rational ($I_{\rm rat} > 0.6$), we select the herding
interaction rule \eqref{2.inter2}. 
If the majority of the market participants is irrational
($I_{\rm rat}<0.4$), we choose the interaction rule \eqref{2.inter1}.
In the intermediate case, the choice of the interaction rule is random.
Clearly, this choice could be refined by relating it to the value of the
ratio $I_{\rm rat}/I_{\rm irr}$. The pairs of individuals
that interact are chosen randomly and at each step all the agents interact
with the background and with another randomly chosen agent, respectively.

After the interaction part, we need to distribute the function $f$ on the grid.
The distribution at $w^*$ is defined by $f(w^*)=f(w)-f(v)$. Then the
part $f(w^*)$ is distributed proportionally to the neighboring grid points
$w_j$ and $w_{j+1}$.
In order to avoid that the post-interaction values become negative, some
restriction on the random variables are needed; we refer to 
\cite[Section~2.1]{DuWo15} for details.

At each time step, we solve the transport part
$$
  \pa_t f = (\Phi(x,w)f)_x
$$
using a flux-limited Lax-Wendroff/upwind scheme. More precisely,
let $\triangle x=1/N$ be the step size for the rationality variable, and recall that
$\triangle t=10^{-5}$ is the time step size.
The value $f(x_i,w_j,t_k)$ is approximated by $f_{i}^k$ for a fixed $w_j$. 
We recall that the upwind scheme reads as
$$
  f_{i}^{k+1} = \left\{\begin{array}{ll}
	f_{i}^k - \frac{\triangle t}{\triangle x}\Phi(x_i,w_j)(f_{i}^k-f_{i-1}^k)
	&\quad\mbox{if }\Phi(x_i,w_j)>0, \\
	f_{i}^k - \frac{\triangle t}{\triangle x}\Phi(x_i,w_j)(f_{i+1}^k-f_{i}^k)
	&\quad\mbox{if }\Phi(x_i,w_j)\le 0,
	\end{array}\right.
$$
and the Lax-Wendroff scheme is given by
$$
  f_i^{k+1} = f_i^k - \frac{\triangle t}{2\triangle x}\Phi(x_i,w_j)
	(f^k_{i+1}-f_{i-1}^k)
	+ \frac{(\triangle t)^2}{2(\triangle x)^2}\Phi(x_i,w_j)^2
	(f_{i+1}^k-2f_i^k+f_{i-1}^k).
$$
The Lax-Wendroff scheme has the advantage that it is of second order, while
the first-order upwind scheme is employed close to discontinuities. 
The choice of the scheme depends on the smoothness of the data. 
In order to measure the smoothness, we compute the ratio $\theta_i^k$ of the
consecutive differences and introduce a smooth van-Leer limiter function
$\Psi(\theta_i^k)$, defined by
$$
	\Psi(\theta_i^k) = \frac{|\theta_i^k|+\theta_i^k}{1+|\theta_i^k|}, 
	\quad\mbox{where }\theta_i^k = \frac{f_i^k-f_{i-1}^k}{f^k_{i+1}-f_i}.
$$
Our final scheme is defined by
\begin{align*}
  f_i^{k+1} & = f_i^{k} - \frac{\triangle t}{\triangle x}\Phi(x_i,w_j)
	(F_{i+1}^{k}-F_i^{k}), 
	\quad\mbox{where} \\
	F_i^{k} &= \frac12(f_{i-1}^{k}-f_i^{k}) - \frac12\mbox{sgn}(\Phi(x_i,w_j))
	\bigg(1-\Psi(\theta_i)\bigg(1-\frac{\triangle t}{\triangle x}
	|\Phi(x_i,w_j)|\bigg)\bigg)(f_i^{k}-f_{i-1}^{k}).
\end{align*}

%%%%%%%%%%%%%%%%%%%%%%%%

\subsection{Choice of functions and parameters}\label{sec.choice}

We still need to specify the functions used in the simulations.
We take $\tau_H=\tau_I=1$, 
$$
  P(|w-W|)=1, \quad d(w)=4w(1-w), \quad \gamma(v,w)=1_{\{w<v\}}v(1-w),
$$
and $\Phi(x,w)$ is given by \eqref{2.Phi}.
%with $\delta=2$ and $\kappa=1$ (if not stated otherwise).
The values of the parameters $\alpha$, $\beta$, $R$, $W$, $\delta$, and $\kappa$
are specified below. 
With the simple setting $P=1$, the interaction rule \eqref{2.inter1} becomes
$w^*=(1-\alpha)w + \alpha W + \eta d(w)$. This means that $\alpha$ measures the
influence of the public source: if $\alpha=1$, the agent adopts the asset value
$W$, being the background value; if $\alpha=0$, the agent is not influenced
by the public source. The random variables $\eta$ is normally distributed
with zero mean and standard deviation $0.06$.

The diffusion coefficient $d(w)$ is chosen such that it vanishes at the
boundary of the domain of definition of $w$, i.e.\ at $w=0$ and $w=1$, and
that its maximal value is one. 

The choice of $\gamma(v,w)$ is similar to that one in \cite[Formula (11)]{DeLo14},
and we explained its structure already in Section \ref{sec.pih}. In
\eqref{2.gamma}, we have chosen $f(w)=1-w$. This means that agents do not
change their asset value due to herding when $w$ is close to its maximal value.
When the asset value is very low, $w\approx 0$, we have 
$w^*\approx \beta v + \eta_1 d(w)$, and the agent adopts
the value $\beta v$. 

%%%%%%%%%%%%%%%%%%%%%%%%

\subsection{Numerical test 1: constant $R$, constant $W$}

We choose $R=0.025$ and $W=0.5$. The aim is to understand the occurence of
bubbles and crashes depending on the parameters $\alpha$, $\beta$, and $\kappa$. 
We say that a bubble (crash) occurs at time $t$ if the mean asset value $m_w(f(t))$
is larger than $W+R$ (smaller than $W-R$). 
This definition is certainly a strong simplification. However,
there seems to be no commonly accepted scientific definition or classification
of a bubble. Shiller \cite[page 2]{Shi15} defines ``a speculative bubble as a situation
in which news of price increases spurs investor enthusiasm, which spreads by
psychological contagion from person to person''. 
Our definition may be different from the usual perception of a bubble or 
crash in real markets.

Figure \ref{fig.bubble} (left) presents the percentage of bubbles and crashes
for different values of $\alpha$. 
More precisely, we count how often the mean asset value is larger
than $W+R$ (smaller than $W-R$) and how often it lies in the range $[W-R,W+R]$. 
The quotient defines the percentage of bubbles (crashes). 
The simulations were performed 200 times
and the mean asset value is then averaged. We observe that
bubbles occur more frequently when $\alpha$ is close to zero. This may be
explained by the fact that $\alpha$ represents the reliability of the public
information, and when this quantity is small, the agents do not trust the
public source. If $\alpha$ is close to one, all the market participants rely
on the public information. This means that they assume an asset value close
to the ``fair'' prize $W$. This corresponds to a herding behavior, and the
herding interaction rule, which tends to higher values, applies, leading
to bubble formation. A market that does neither overestimate nor underestimate
public information leads to the smallest bubble percentage, here with
$\alpha$ being around $0.5$. Interestingly, the results vary only slightly
with repect to the parameter $\beta$.

\begin{figure}[ht]
\centering
\includegraphics[width=80mm]{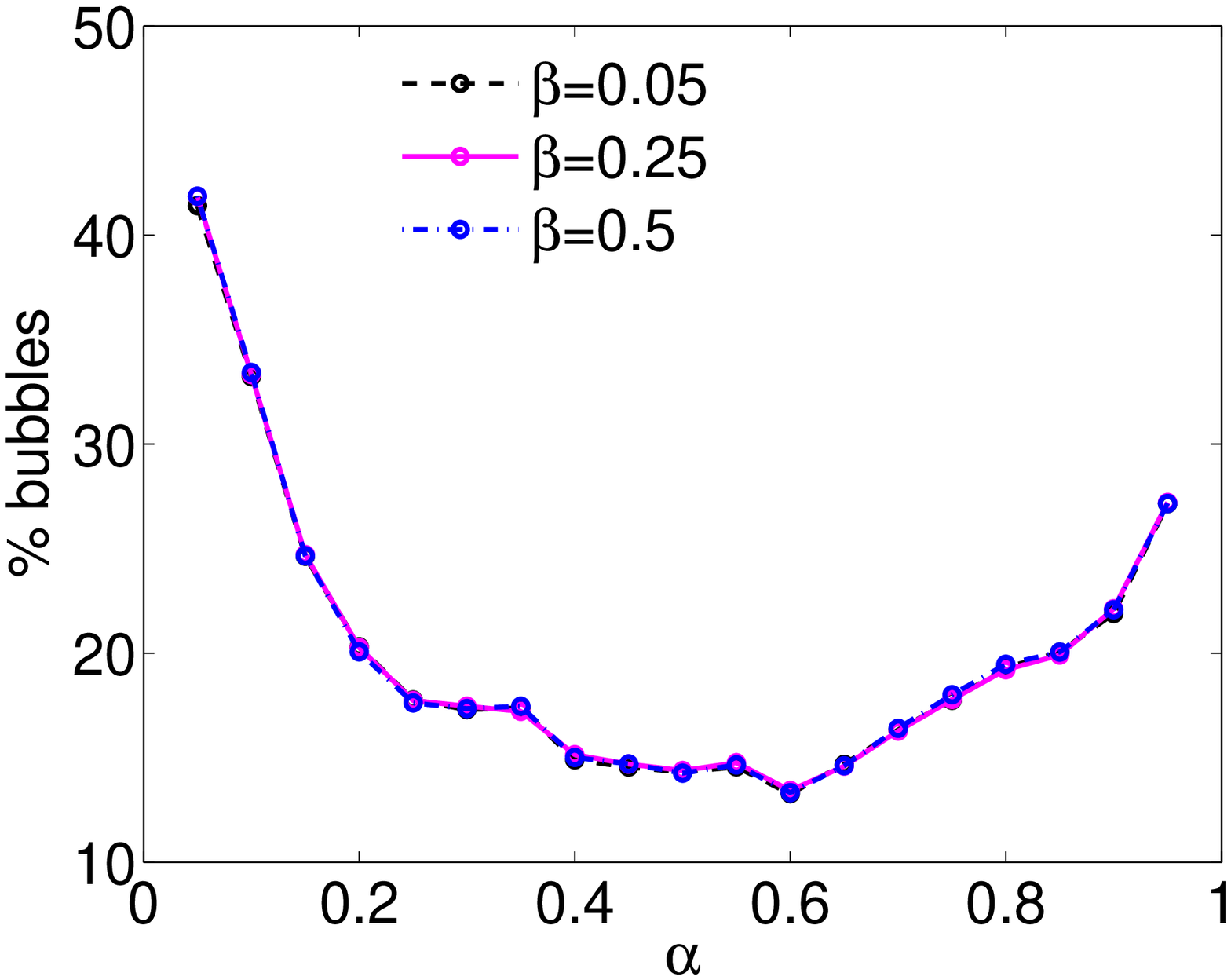}
\includegraphics[width=80mm]{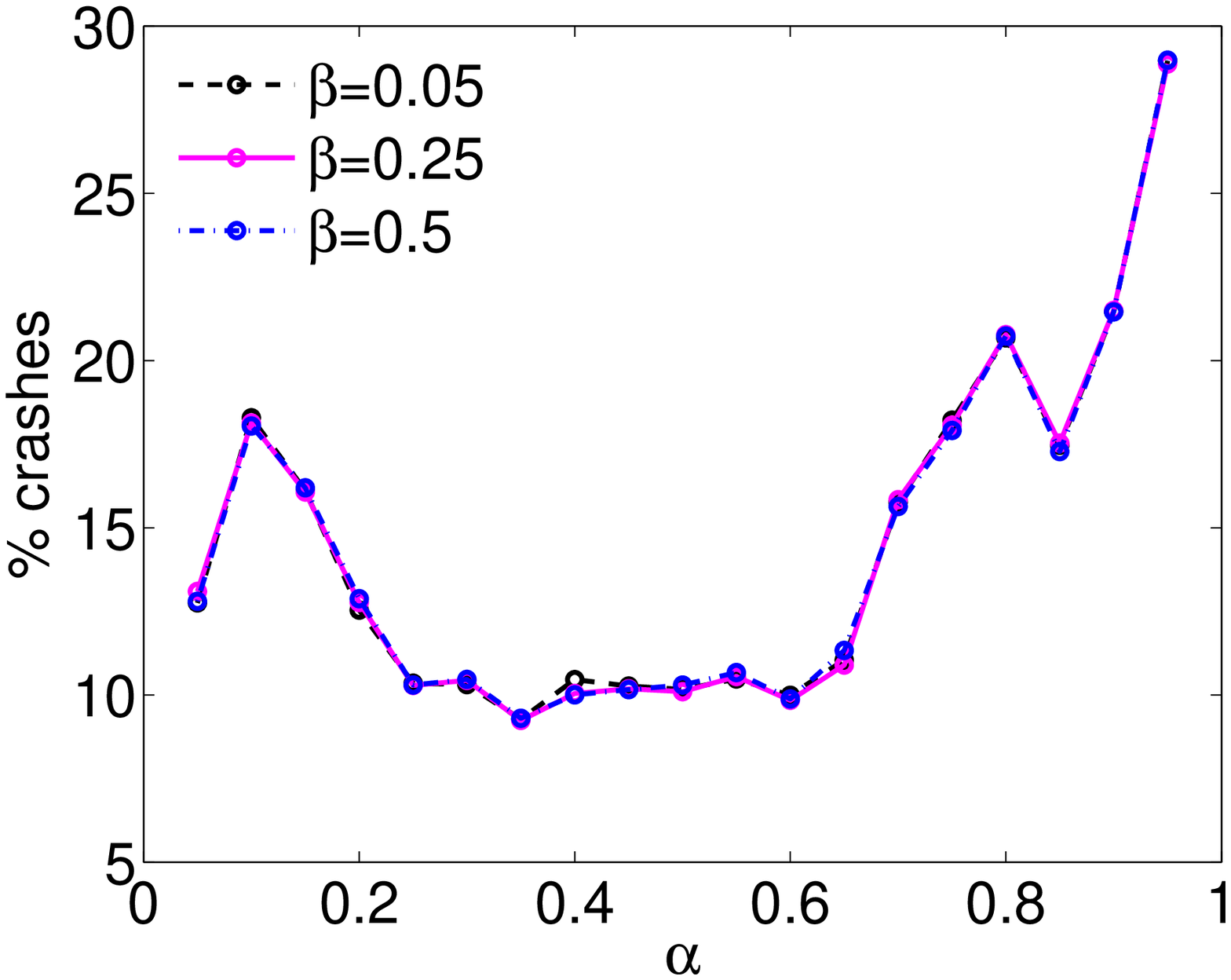}
\caption{Left: Percentage of bubbles (left) and crashes (right)
depending on the choice of $\alpha$ and $\beta$. The parameters are
$R=0.025$, $W=0.5$, $\delta=\kappa=1$.}
\label{fig.bubble}
\end{figure}

The percentage of crashes is depicted in Figure \ref{fig.bubble} (right).
Qualitatively, the percentage is small for values $\alpha$ not too far from 0.5,
but the shape of the curves is more complex than those for bubbles. For instance,
there is a local maximum at $\alpha=0.1$ and a local minimum at $\alpha=0.85$.
The percentage of crashes is largest for $\alpha$ close to one. Again,
the dependence on the parameter $\beta$ is very weak.

In the above simulations, we have assumed a constant value for $\alpha$, i.e.,
all market participants have the same attitude to change their mind when
interacting with public sources, We wish to show that nonconstant values
lead to similar conclusions. For this, we generate $\alpha$ from a normal
distribution with standard deviation 0.45 and various means $\langle\alpha\rangle$.
The result is shown in Figure \ref{fig.alpha} for $\beta=0.25$ and $\beta=0.5$. 
For comparison, the percentages for constant $\alpha$ and $\beta=0.05$ are also shown.
It turns out that the results for nonconstant or constant $\alpha$ are
qualitatively similar which justifies the use of constant $\alpha$.

\begin{figure}[ht]
\centering
\includegraphics[width=90mm]{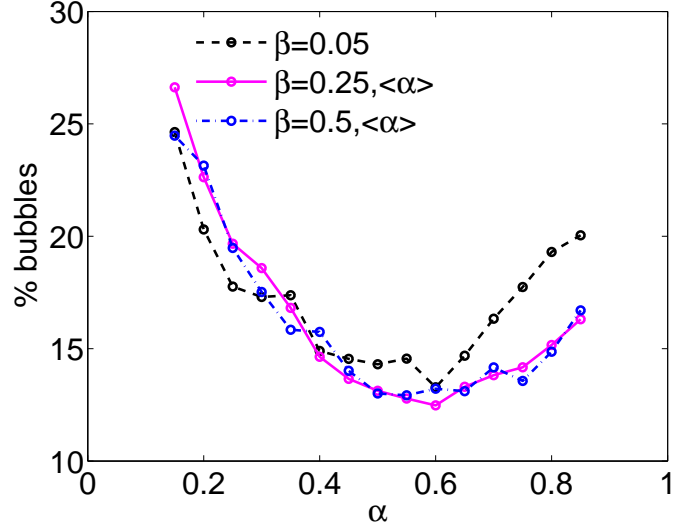}
\caption{Percentage of bubbles for varying $\alpha$ ($\beta=0.25$) and
constant $\alpha$ ($\beta=0.5$ and $\beta=0.05$).}
\label{fig.alpha}
\end{figure}

%%%%%%%%%%%%%%%%%%%%%%%%

\subsection{Numerical test 2: constant $R$, time-dependent $W(t)$}

Now, we chose $R=0.025$ and 
$$
  W(t) = 0.1 + 0.05\bigg(\sin\frac{t}{500\triangle t} 
	+ \frac12\exp\frac{t}{1500\triangle t}\bigg), \quad t\ge 0.
$$
The time evolution of the first moment $m_w(f(t))=\int_{\Omega_1} f(x,w,t)wdz$
is shown in Figure~\ref{fig.meanw}. We see that the mean asset value stays
within the range $[W(t)-R,W(t)+R]$ if $\alpha$ is small (except for increasing
``fair'' prices) and it has the
tendency to take values larger than $W(t)$ if $\alpha$ is large.

\begin{figure}[ht]
\centering
\includegraphics[width=80mm]{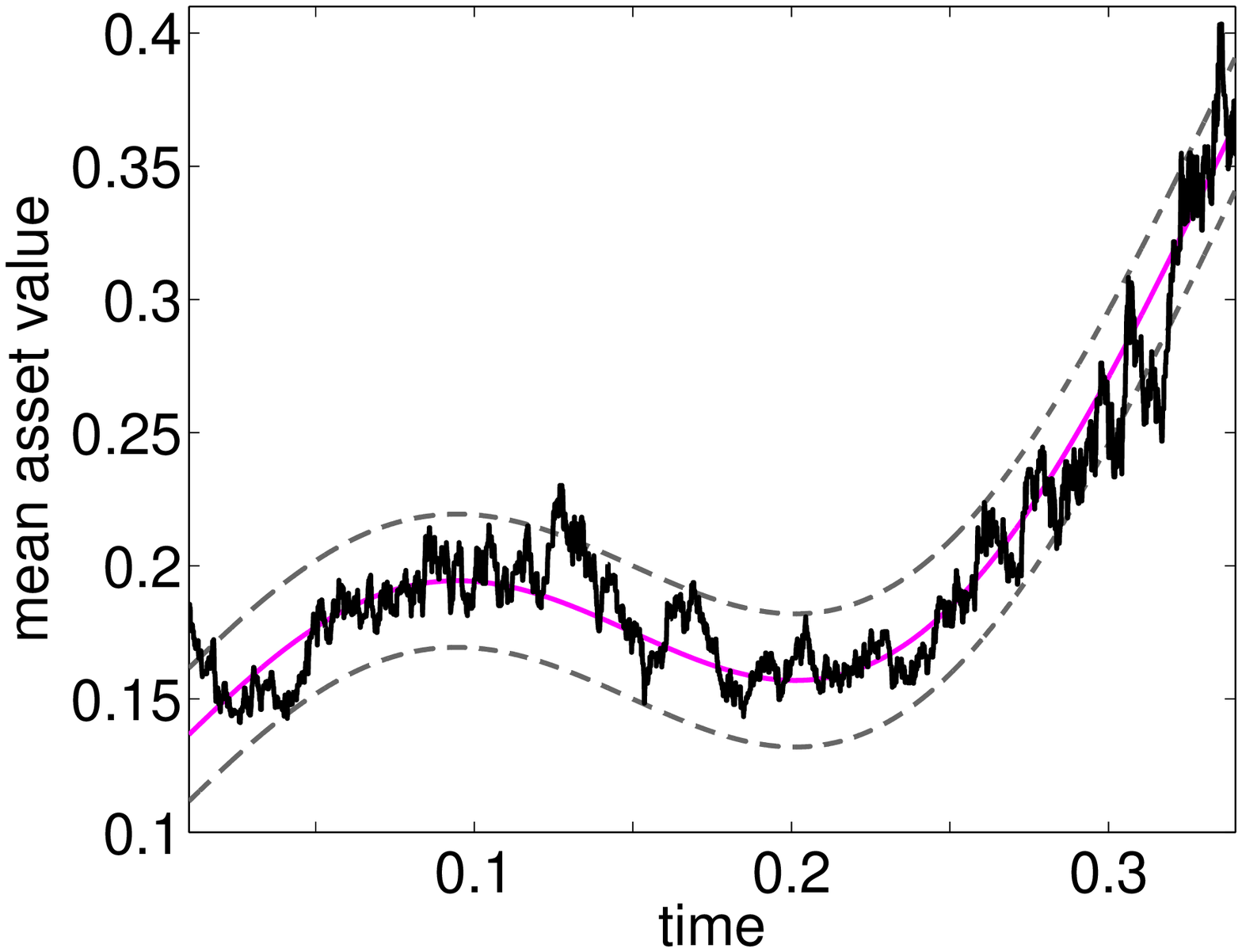}
\includegraphics[width=80mm]{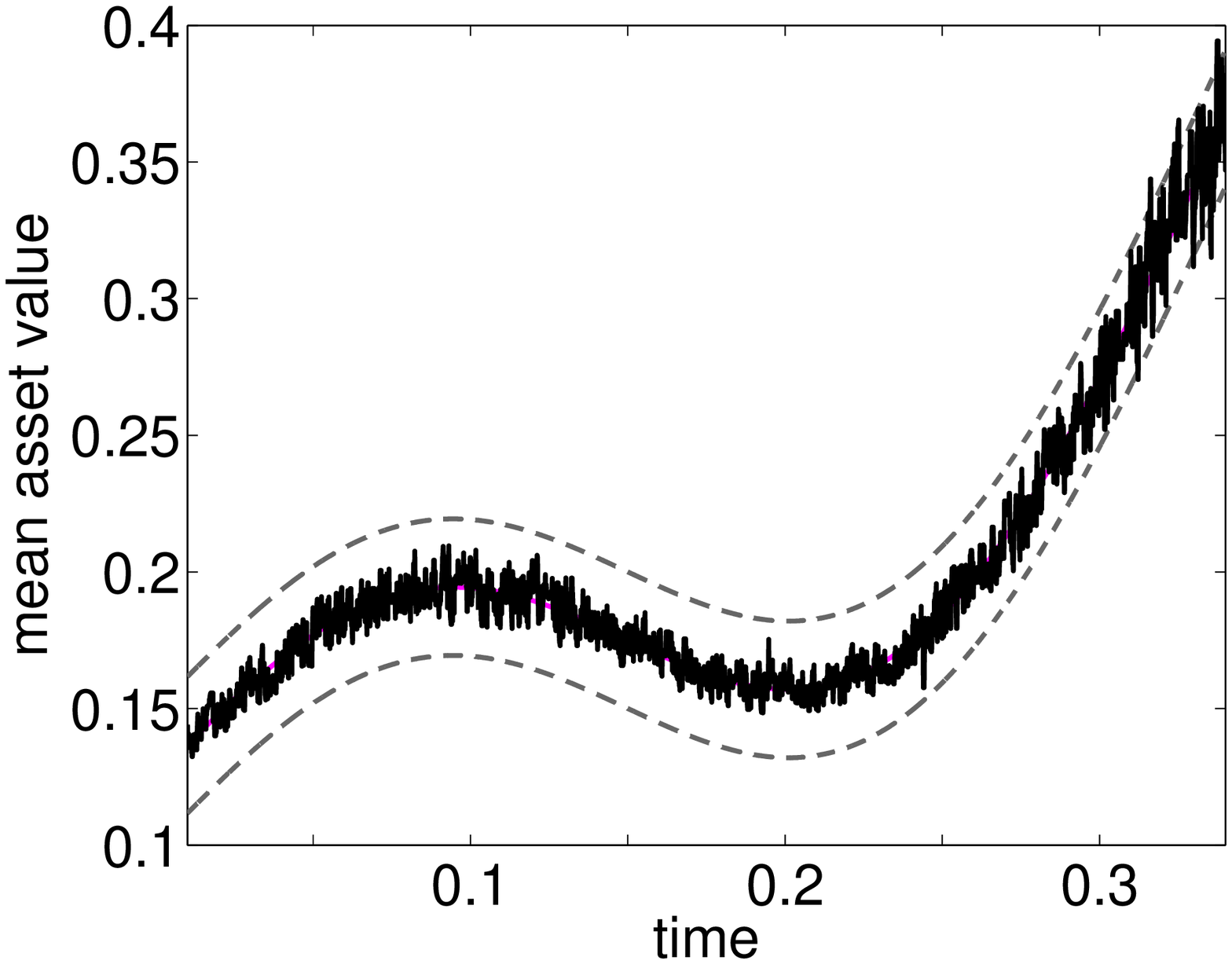}
\caption{Mean asset value $m_w(f(t))$ versus time $t$ for $\alpha=0.5$
(left) and $\alpha=0.05$ (right). The function
$W(t)$ is represented by the solid line in between the dashed lines which
represent the functions $W(t)+R$ and $W(t)-R$. The parameters are 
$\beta=0.25$, $R=0.025$, $\delta=2$, $\kappa=1$.}
\label{fig.meanw}
\end{figure}

Figure \ref{fig.delta} illustrates the influence of the parameter $\delta$
which describes the strength of the drift in the region $|w-W(t)|<R$.
The background value $W(t)$ models a crash: It increases up to time $t=0.2$
then decreases abruptly, and stays constant for $t>0.2$. 
For small values of $\delta$,
the mean asset value decreases slowly while it adapts to $W(t)$ more
quickly when $\delta$ is large. Interestingly, we observe a (small) time delay
for small $\delta$ although the equations do not contain any delay term.
The delay is only caused by the slow drift term.
The same phenomenon can be reproduced for abruptly increasing $W(t)$.

\begin{figure}[ht]
\centering
\includegraphics[width=80mm]{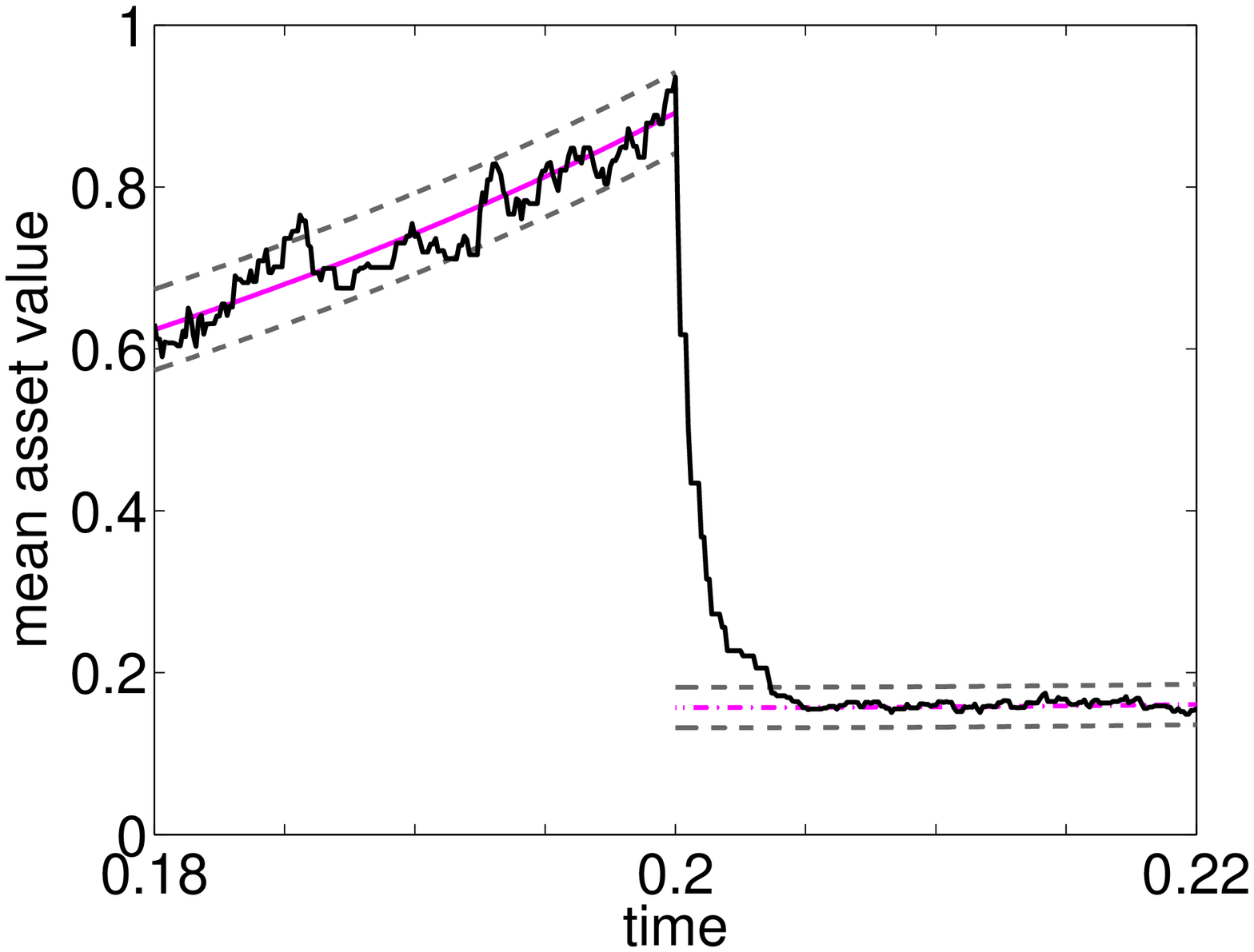}
\includegraphics[width=80mm]{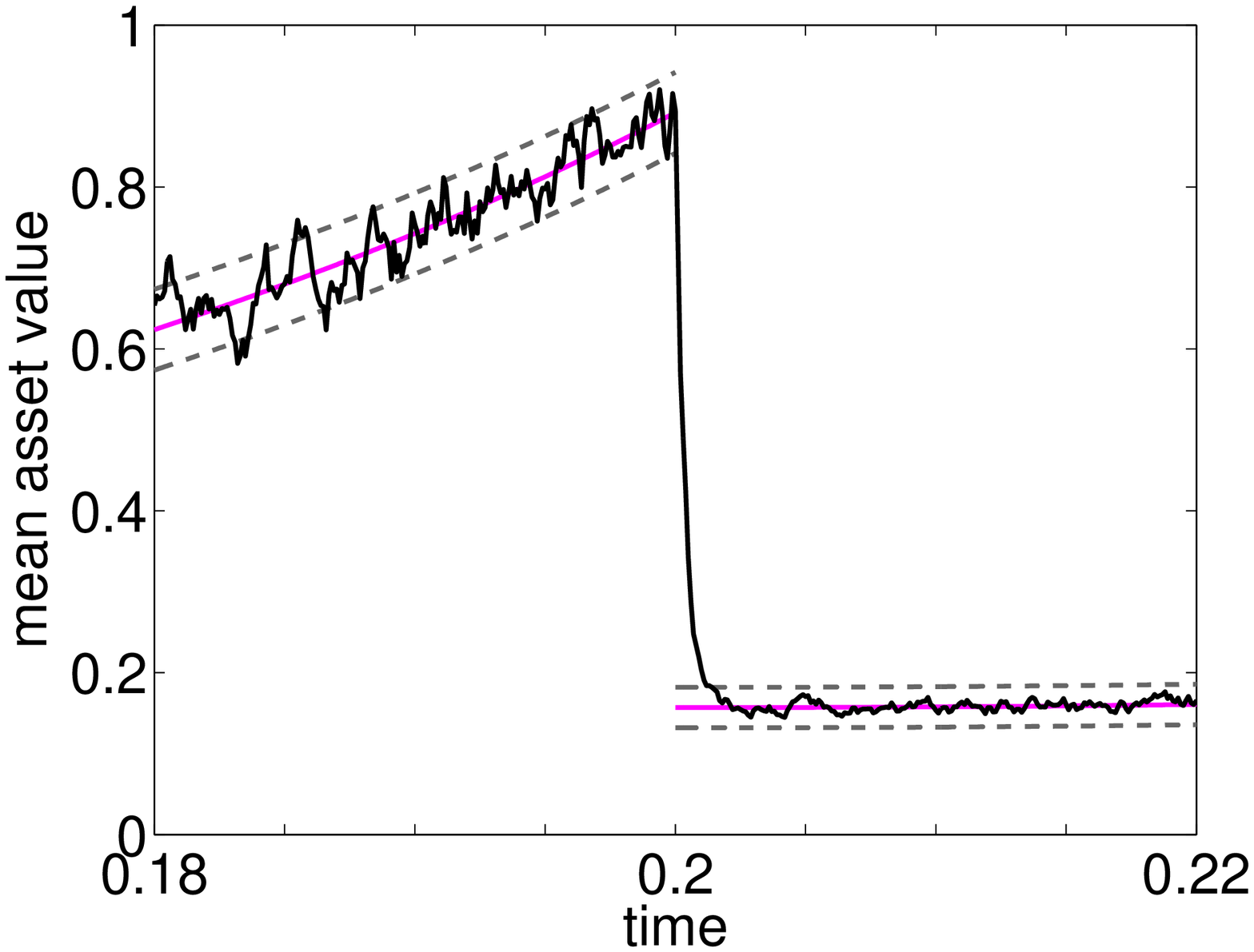}
\caption{Mean asset value $m_w(f(t))$ versus time for $\delta=0.01$ 
(left) and $\delta=100$ (right) with $\alpha=0.25$, $\beta=0.2$, $R=0.025$,
$\kappa=1$.}
\label{fig.delta}
\end{figure}

%%%%%%%%%%%%%%%%%%%%%%%%%%

\subsection{Numerical test 3: time-dependent $R(t)$}

The final numerical test is concerned with time-dependent bounds $R(t)$.
We distinguish the upper and lower bound and accordingly the
boundaries $w=W(t)+R^+(t)$ and $w=W(t)-R^-(t)$. 
The functions $R^\pm(t)$ are defined as the Bollinger bands which
are volatility bands above and below a moving average. They are employed
in technical chart analysis although its interpretation may be delicate.
The definition reads as
$$
  R^\pm(t_k) = M_n(t_k) \pm k\sigma(t_k),
$$
where $M_n(t_k)$ is the $n$-period moving average (we take $n=30$), 
$k$ is a factor (usually $k=2$), and $\sigma(t_k)$ is the 
corrected sample standard deviation,
$$
  M_n(t_k) = \frac{1}{n}\sum_{\ell=1}^n m_w(f(t_{k-\ell})), \quad
	\sigma(t_k) = \bigg(\frac{1}{n-1}\sum_{\ell=1}^n
	\big(m_w(f(t_{k-\ell}))-M_n(t_{k-\ell})\big)^2\bigg)^{1/2}.
$$

Figure \ref{fig.boll1} shows the time evolution of the mean asset value
and the Bollinger bands for two different values of $\alpha$ and constant $W$.
One may say that the market is overbought (or undersold) 
when the asset value is close to the upper (or lower) Bollinger band.
For small values of $\alpha$, the market participants are not much
influenced by the public information and they tend to increase their estimated
asset value due to herding. 

\begin{figure}[ht]
\centering
\includegraphics[width=80mm]{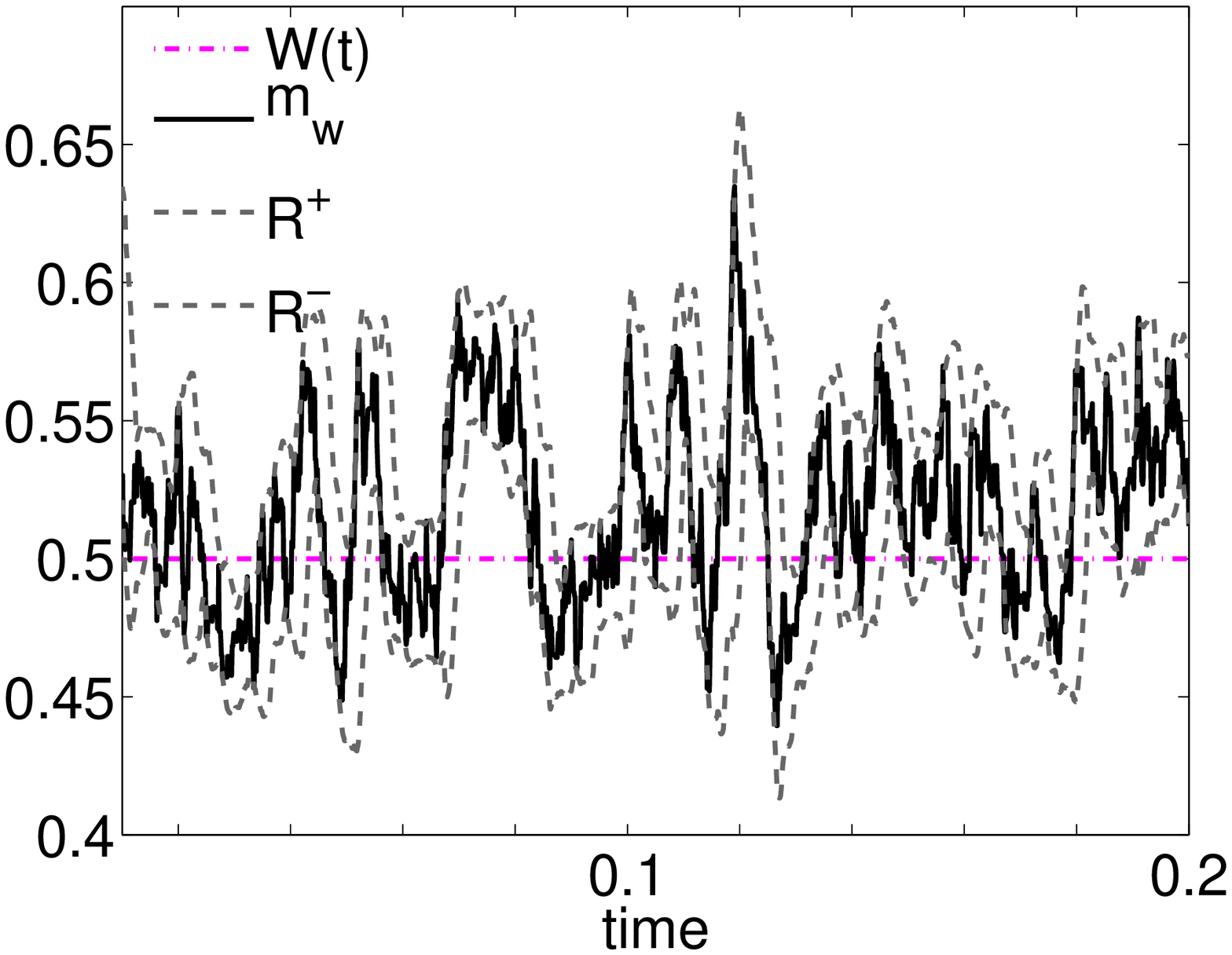}
\includegraphics[width=80mm]{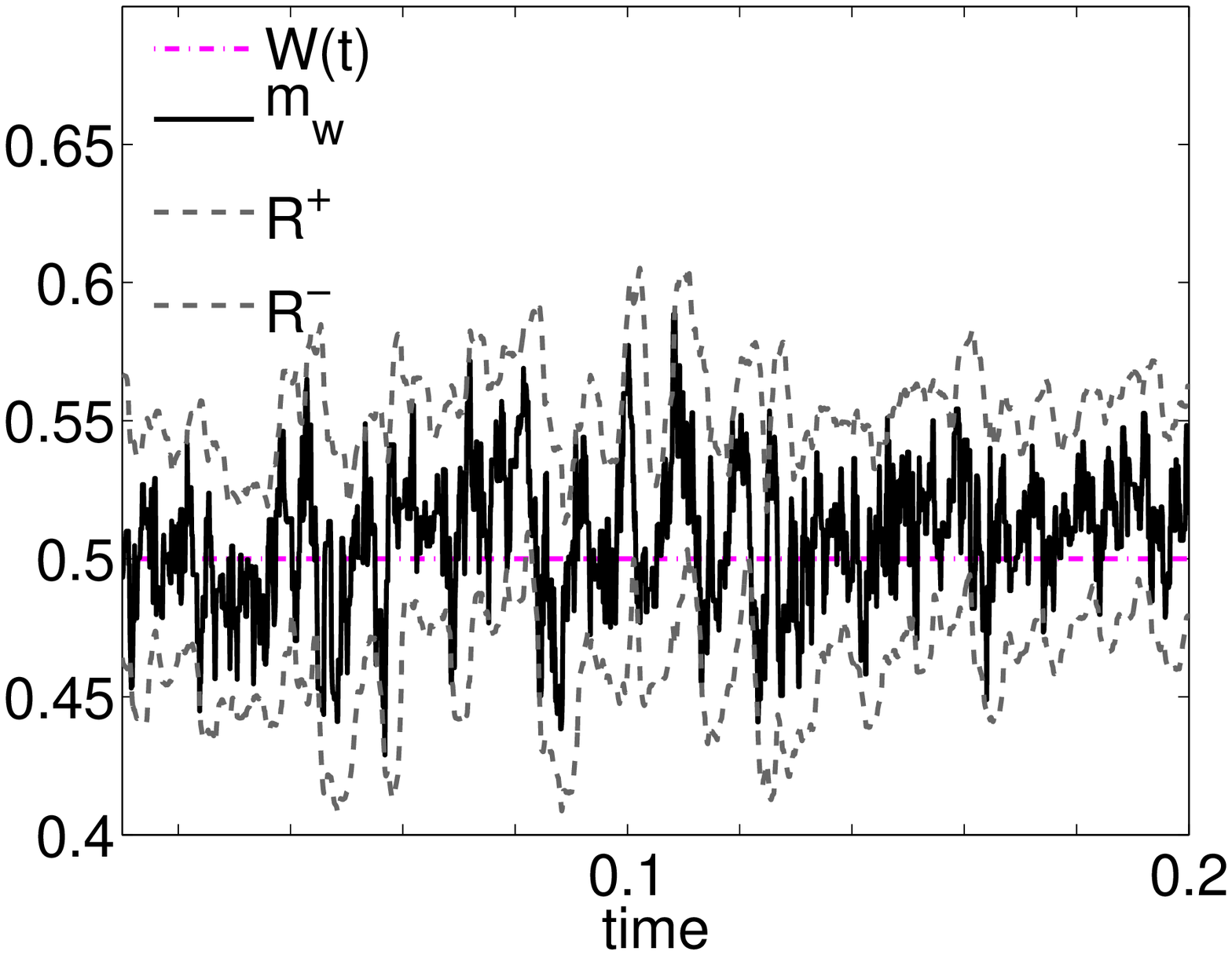}
\caption{Mean asset value $m_w(f(t))$ and Bollinger bands $R^\pm(t)$ versus 
time for $\alpha=0.2$ (left) and $\alpha=0.05$ (right). The parameters are
$\beta=0.25$, $W=0.5$, $\delta=1$, $\kappa=1$.}
\label{fig.boll1}
\end{figure}

The mean asset value and the corresponding Bollinger bands for a discontinuous
background value $W(t)$ is displayed in Figure \ref{fig.boll2} (left column).
We have chosen $d(w)=w(1-w)$ and $\eta=\pm 0.06$ (upper row) or
$\eta=\pm 0.18$ (lower row). 
The value $W(t)$ abruptly decreases at time $t=0.2$. 
We are interested in the difference of the upper and lower Bollinger bands,
more precisely in the Bollinger bandwidth $B(t)=100(R^+(t)-R^-(t))/W(t)$,
measuring the relative difference between the upper and lower Bollinger bands.
According to chart analysts, falling (increasing) bandwidths reflect decreasing
(increasing) volatility. 
In our simulation, the jump of $W(t)$ gives rise to a peak of the Bollinger 
bandwidth at $t=0.2$; see Figure \ref{fig.boll2} (right column). 
Another small peak can be observed at $t\approx 0.38$ (upper right figure)
when $\eta=\pm 0.06$. For larger values of $\eta$ (lower right figure),
the fluctuations in the Bollinger bandwidth are larger.

In chart analysis, the bandwidth is employed to identify
a band squeeze. When the asset value leaves the interval $[R^-,R^+]$, this situation
may indicate a change of direction of the prices. Clearly, this interpretation
cannot be directly applied to the present situation. On the other hand, 
the Bollinger bands are an additional tool to identify large changes in the
mean asset value, for instance when the background value $W(t)$ is no longer
deterministic but driven by some stochastic process. We leave this
generalization for future work.

\begin{figure}[ht]
\centering
\includegraphics[width=80mm]{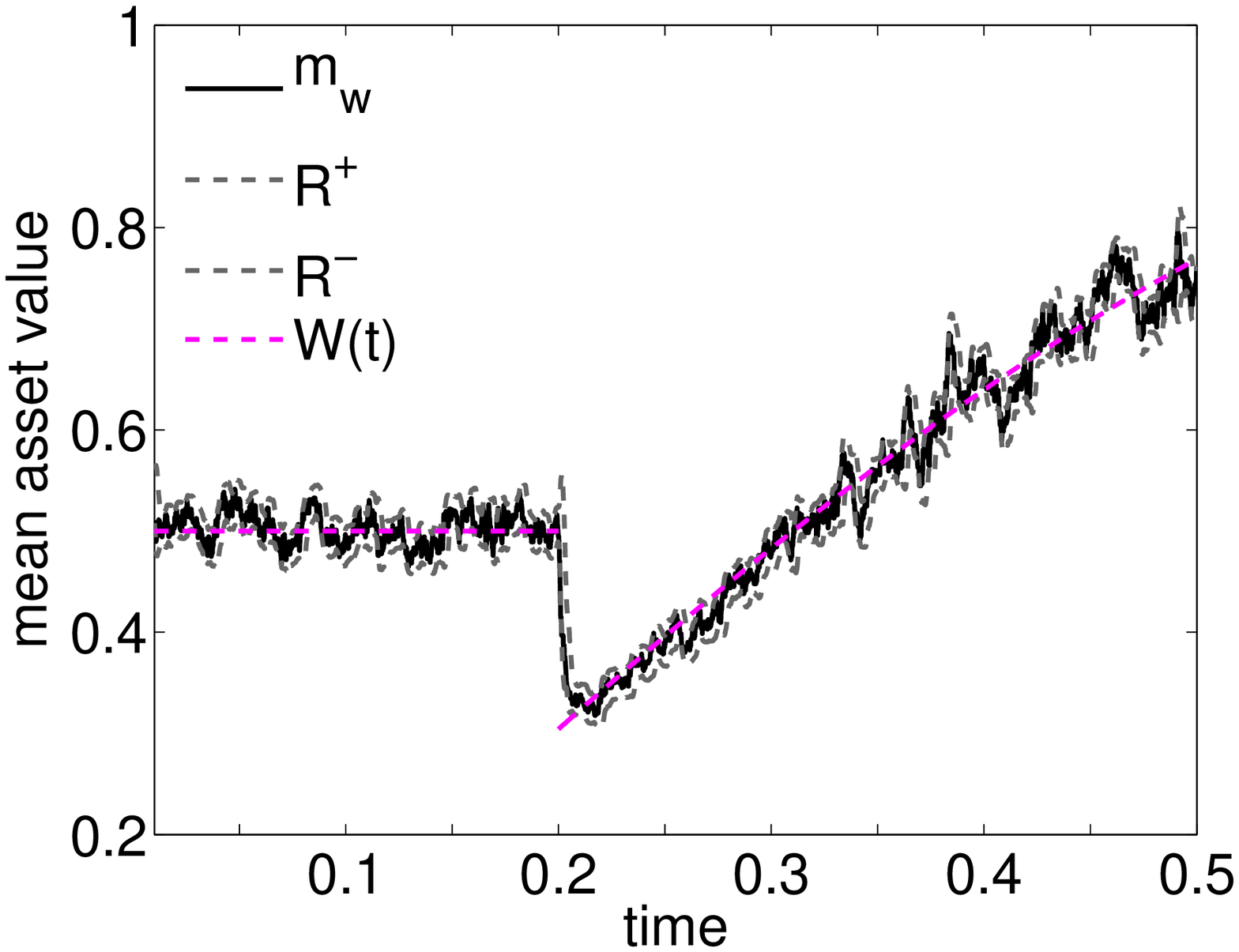}
\includegraphics[width=80mm]{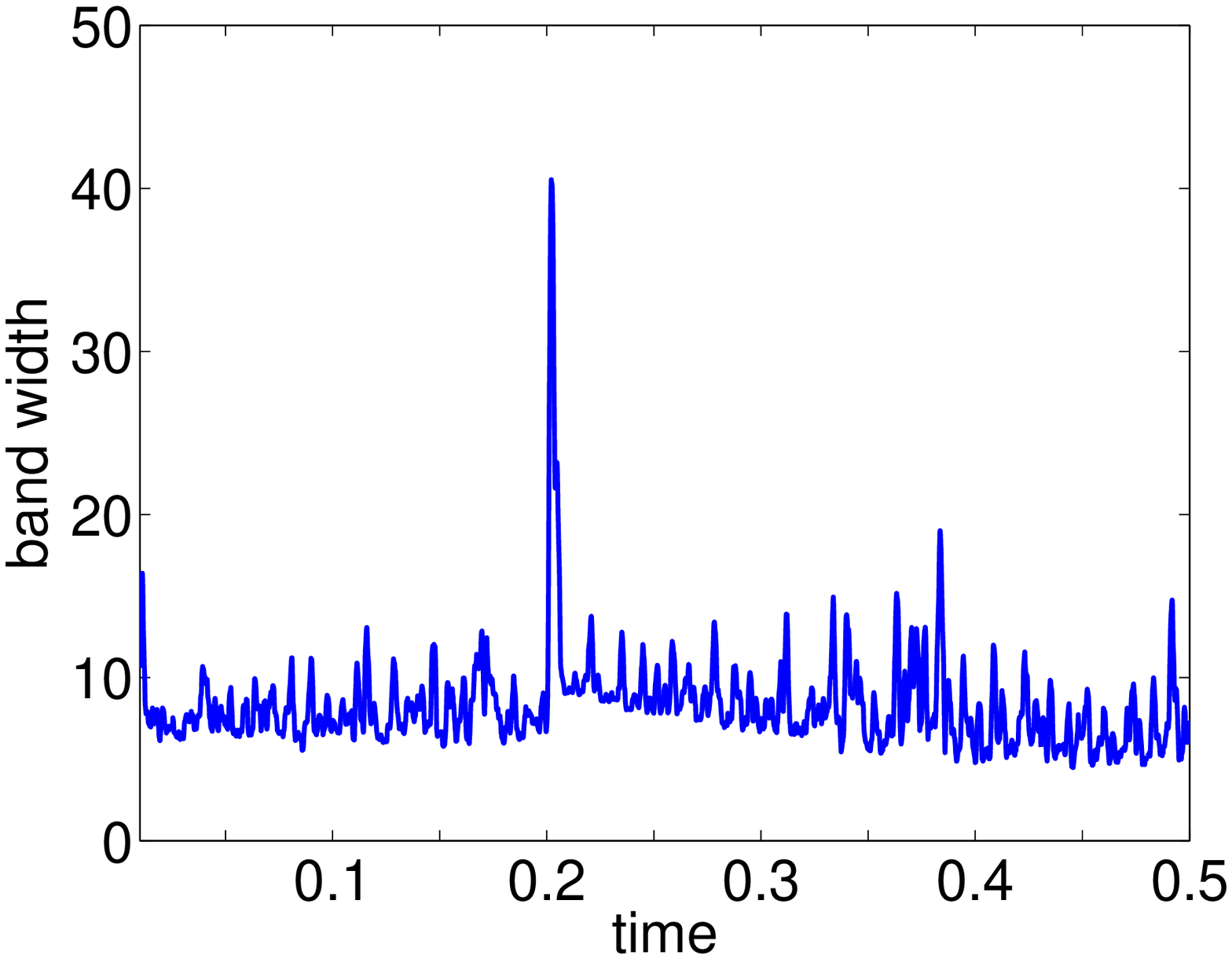}
\includegraphics[width=80mm]{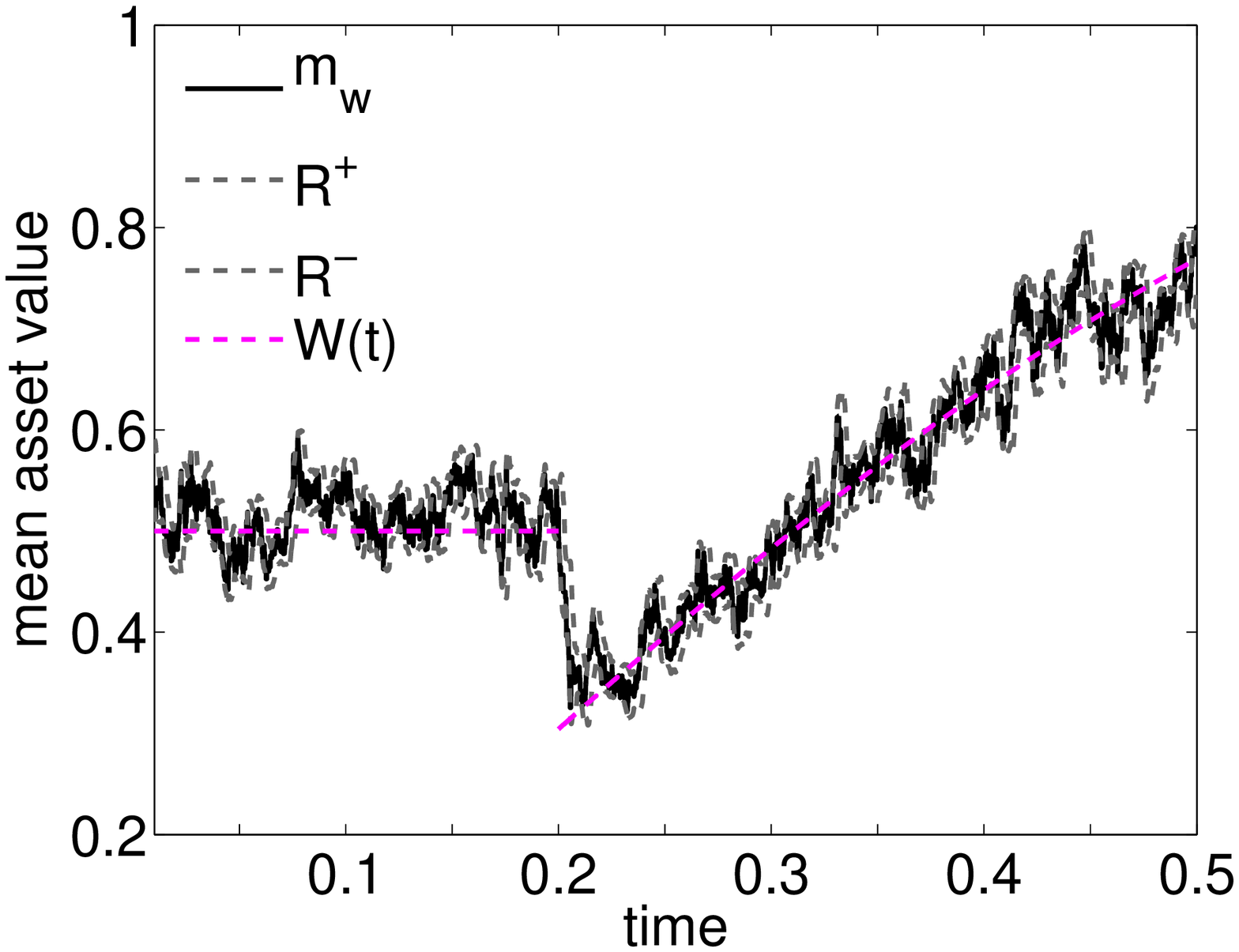}
\includegraphics[width=80mm]{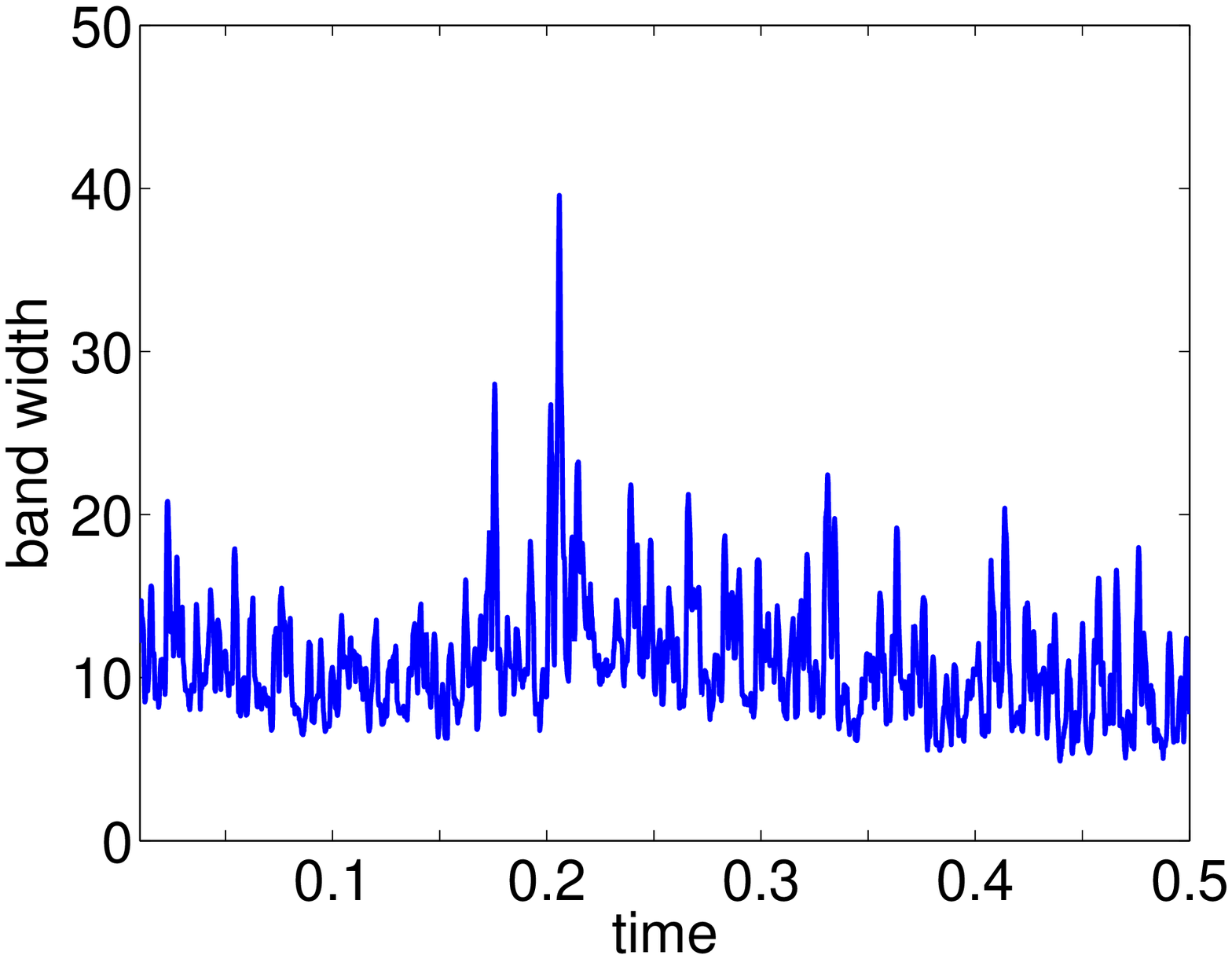}
\caption{Mean asset value $m_w(f(t))$ (left column) and Bollinger bands 
$R^\pm(t)$ (right column) versus time. The function $W(t)$ has a jump at $t=0.2$. The
parameters are $\alpha=0.05$, $\beta=0.25$, $R=0.025$, $\delta=\kappa=1$.
Upper row: $\eta=\pm 0.06$, lower row: $\eta=\pm 0.18$.}
\label{fig.boll2}
\end{figure}

%%%%%%%%%%%%%%%%%%%%%%%%%%%%%%%%%%%%%%%%%%%%%%%%%%%%%%%%%%%%%%%%%%%%%%%%%%%%%

\end{document}